\documentclass[10pt]{article}
\usepackage[nohead,margin=1.0in]{geometry}
\usepackage{amssymb, amsmath, amsthm}
\usepackage{color}
\usepackage{graphicx}
\usepackage{subfigure}
\usepackage{epstopdf}
\graphicspath{{./pics/}}
\usepackage{appendix}
\usepackage{booktabs}
\usepackage{verbatim}
\numberwithin{equation}{section}
\newtheorem{theorem}{Theorem}[section]
\newtheorem{remark}{Remark}[section]
\newtheorem{proposition}{Proposition}[section]
\newtheorem{lemma}{Lemma}[section]
\newtheorem{corollary}{Corollary}[section]

\newtheorem{assumption}[theorem]{Assumption}
\newtheorem{example}{Example}[section]

\allowdisplaybreaks

\def\II{(\Omega)}
\title{Convergence Rate Analysis of Galerkin Approximation of Inverse Potential Problem\thanks{The work of B. Jin is supported by UK EPSRC grant EP/T000864/1 and EP/V026259/1, X. Lu by the National Science Foundation of China (No. 11871385), and that of Z. Zhou by Hong Kong Research Grants Council grant (No. 15303021)
and an internal grant of The Hong Kong Polytechnic University (Project ID: P0031041, Work Programme: 4-ZZKS).}}

\author{Bangti Jin\thanks{Department of Mathematics, The Chinese University of Hong Kong, Shatin, New  Territories, Hong Kong, P.R. China.
(\texttt{btjin@math.cuhk.edu.hk, b.jin@ucl.ac.uk, bangti.jin@gmail.com})} \and Xiliang Lu\thanks{School of Mathematics and Statistics, and Hubei Key Laboratory of Computational Science, Wuhan University, Wuhan 430072, People's
 Republic of China (\texttt{xllv.math@whu.edu.cn})} \and Qimeng Quan\thanks{School of Mathematics and Statistics, Wuhan University, Wuhan 430072, People's
 Republic of China (\texttt{quanqm@whu.edu.cn})}  \and Zhi Zhou\thanks{Department of Applied Mathematics,
The Hong Kong Polytechnic University, Kowloon, Hong Kong (\texttt{zhizhou@polyu.edu.hk})}}

\begin{document}

\maketitle

\begin{abstract}
In this work we analyze the inverse problem of recovering a space-dependent potential coefficient in an elliptic / parabolic problem from distributed observation. We establish novel (weighted) conditional stability estimates under very mild conditions on the problem data. Then we provide an error analysis of a standard reconstruction scheme based on the standard output least-squares formulation with Tikhonov regularization (by an $H^1$-seminorm penalty), which is then discretized by the Galerkin finite element method with continuous piecewise linear finite elements in space (and also backward Euler method in time for parabolic problems). We present a detailed error analysis of the discrete scheme, and provide convergence rates in a weighted $L^2(\Omega)$ for discrete approximations with respect to the exact potential. The error bounds explicitly depend on the noise level, regularization parameter and discretization parameter(s). Under suitable conditions, we also derive error estimates in the standard $L^2(\Omega)$ and interior $L^2$ norms. The analysis employs sharp a priori error estimates and nonstandard test functions. Several numerical experiments are given to complement the theoretical analysis.\vskip10pt
\textbf{Keywords}: inverse problems, parameter identification, Tikhonov regularization, error estimate
\end{abstract}

\section{Introduction}
In this work, we study the inverse problem of recovering a space-dependent potential coefficient in elliptic and parabolic equations. Let $\Omega\subset \mathbb{R}^d\ (d=1,2,3)$ be a simply connected convex polyhedral domain with a boundary $\partial\Omega$. Then the governing equation in the elliptic and parabolic cases are given respectively by 
\begin{equation}\label{equ:elliptic problem}
\left\{\begin{aligned}
		-\Delta u +qu &= f, \ &\mbox{in}&\ \Omega, \\
		u&=0, \ &\mbox{on}&\ \partial\Omega,
\end{aligned}\right.
\end{equation}
and
\begin{equation}\label{equ:parabolic problem}
	\left\{
	\begin{aligned}
		\partial_tu-\Delta u +qu &= f, \ &\mbox{in}&\ \Omega\times(0,T), \\
		u&=0, \ &\mbox{on}&\ \partial\Omega\times(0,T), \\
		u(\cdot,0)&=u_0, \ &\mbox{in}&\ \Omega,
	\end{aligned}
	\right.
\end{equation}
where $T>0$ is the final time. The functions $f$ and $u_0$ in \eqref{equ:elliptic problem} and \eqref{equ:parabolic problem} are the source and initial data, respectively. The space-dependent potential $q$ belongs to the admissible set $K$ such that
\begin{equation*}
    K = \{q\in L^\infty(\Omega): c_0\leq q(x)\leq c_1 \mbox{ a.e. in } \Omega\},
\end{equation*}
with $0\leq c_0<c_1<\infty$.
To explicitly indicate the dependence of the solution $u$ to problems \eqref{equ:elliptic problem} and \eqref{equ:parabolic problem} on the potential $q$, we write $u(q)$. Further, we are given the observational data $z^\delta$ on $\Omega$ or $\Omega\times (T_0,T)$:
\begin{equation*}
   \left\{\begin{aligned}
       z^\delta(x) &= u(q^\dag)(x) + \xi(x), \quad x\in \Omega,&&\mbox{elliptic}\\
       z^\delta(x,t) &= u(q^\dag)(x,t) + \xi(x,t), \quad (x,t)\in \Omega\times(T_0,T),&& \mbox{parabolic}
   \end{aligned}\right. 
\end{equation*}
where $u(q^\dagger)$ denotes the exact data (corresponding to the exact potential $q^\dagger$), $0\leq T_0<T$, and $\xi$ denotes the measurement noise. The accuracy of the data $z^\delta$ is measured by the noise level $\delta=\|u(q^\dagger)-z^\delta\|_{L^2(\Omega)}$ or $\delta=\|u(q^\dagger)-z^\delta\|_{L^2(T_0,T;L^2(\Omega))}$, in the elliptic and parabolic cases, respectively. The inverse potential problem is to recover the potential $q$ from the noisy observation $z^\delta$. It arises in several practical applications, where $q$ represents the radiativity coefficient in heat conduction \cite{YamamotoZou:2001} and perfusion coefficient in Pennes' bio-heat equation in human physiology \cite{Pennes:1948,TrucuInghamLesnic:2010} (see \cite{ScottRobinson:1998,YueZhangZuo:2008} for experimental studies) and the elliptic case also in quantitative dynamic elastography \cite{Choulli:2020}. 

The inverse potential problem is ill-posed, which poses challenges to construct accurate and stable numerical approximations. A number of reconstruction methods have been designed to overcome the ill-posed nature, with the most prominent one being  Tikhonov regularization \cite{Engl1996Regularization,ItoJin:2015}. 
In practical computation, one still needs to discretize the continuous regularized formulation. This is often achieved by the Galerkin finite element method (FEM) when the domain $\Omega$ is irregular and the problem data ($u_0$ and $f$) have only limited regularity. This strategy has been widely used \cite{YamamotoZou:2001,DengYuYang:2008,YangYuDeng:2008}. Yamamoto and Zou \cite{YamamotoZou:2001} proved the convergence of the discrete approximations in the parabolic case. However, the convergence rates of discrete approximations are generally very challenging to obtain, due to the \textit{inherent nonconvexity} of the regularized functional, which itself stems from the high degree of \textit{nonlinearity} of the parameter-to-state map, despite the PDEs \eqref{equ:elliptic problem} and \eqref{equ:parabolic problem} being linear. Indeed, this has been a long standing issue for the numerical analysis of many nonlinear inverse problems, e.g., parameter identifications for PDEs. So far there have been only very few error bounds on discrete approximations, despite the fact that such an analysis would provide useful guidelines for choosing suitable discretization parameters. For the related inverse conductivity problem, the works \cite{Wang2010ErrorEO,Jin2021EA} derived error bounds in a weighted $L^2(\Omega)$ norm by employing a special test function for elliptic and parabolic cases, and the latter work \cite{Jin2021EA}  also gives the standard $L^2(\Omega)$ error estimates with the help of a positivity condition on the weighted function.

In this work we study the concerned elliptic and parabolic inverse potential problems, and contribute in the following two aspects. First, we establish novel conditional stability estimates for the concerned inverse problem, including both weighted $L^2(\Omega)$ and standard $L^2(\Omega)$ stability. The latter is obtained under a certain positivity condition, which can be verified for a class of problem data. The derivation is purely variational, using only a nonstandard test function, and extends directly to the error analysis. Our analysis strategy is similar to that in the interesting works \cite{doi:10.1137/16M1094476,bachmayr2019identifiability}, which are concerned with recovering the diffusion coefficient from the internal measurements.  Second, we derive novel weighted $L^2(\Omega)$ error bounds for the discrete approximations under very mild regularity conditions on the problem data and unknown coefficient $q$ as well as the standard $L^2(\Omega)$ error bounds under some positivity condition. Note that the analysis does not employ standard source type conditions. Instead, it is achieved by a novel choice of the test function in the weak formulation, as in the conditional stability analysis, adapting the stability argument to the discrete setting, which allows us to bypass the standard source condition. To the best of our knowledge, these results represent the first error bounds for the discrete approximations for the inverse potential problem. Further, we provide several numerical experiments to complement the theoretical analysis.

Now we review existing works on the analysis and numerics of the inverse potential problem. Several uniqueness and stability results have been obtained \cite{PrilepkoKostin:1992,ChoulliYamamoto:2008,KamyninKostin:2010,BerettaCavaterra:2011,Choulli:2020}. Choulli and Yamamoto \cite{ChoulliYamamoto:2008} proved the uniqueness of recovering the potential $q^\dag$, initial condition and boundary coefficient from terminal measurement, and also gave a stability result under a smallness condition. In the parabolic case, Beretta and Cavaterra \cite{BerettaCavaterra:2011} proved the unique recovery of the potential $q(x)$ from the time-averaged observation. More recently, Choulli \cite{Choulli:2020} derived a new stability estimate in the elliptic case. The well-posedness of the continuous regularized formulation has been analyzed for both elliptic / parabolic cases \cite{EnglKunischNeubauer:1989,YamamotoZou:2001,DengYuYang:2008,YangYuDeng:2008}, and convergence rates with respect to the noise level $\delta$ were obtained under various conditions. In the 1D elliptic case, Engl et al \cite[Example 3.1]{EnglKunischNeubauer:1989} derived a convergence rate of the regularized approximation by Tikhonov regularization under the standard source condition with a small sourcewise representer. Hao and Quyen \cite{HaoQuyen:2010} presented a different approach without explicitly using the source condition. More recently, Chen et al \cite{Chen2020Convergence} proved conditional stability of the inverse problem in negative Sobolev spaces, which allows deriving variational inequality type source conditions and also showing convergence rates of the regularized solutions, and the authors studied both elliptic and parabolic cases. Klibanov, Li, and Zhang \cite{KlibanovLiZhang:2020} presented an interesting convexification method for the inverse problem which allows proving globle convergence despite the nonlinearity of the inverse problem (and actually the method allows recovering a time-independent source simultaneously). This work extends the current literature with new stability analysis and error analysis of discrete approximations Broadly speaking, the present work is along the line of research which connects stability analysis with error analysis of discrete schemes (see, e.g., \cite{BurmanErn:2021,BurmanOksanen:2020}) and convergence rate with conditional stability (see, e.g., \cite{ChengYamamoto:2000,WernerHofmann:2020}).

The rest of the paper is organized as follows. In Section \ref{sec:stab} we present novel conditional stability estimates for the concerned inverse problems. In Sections \ref{sec:err-ell} and \ref{sec:err-para}, we describe the regularized formulations and their finite element discretizations, and derive novel error bounds on the discrete approximations, for the elliptic and parabolic cases, respectively. 
In Section \ref{sec:numer}, we present one- and two-dimensional numerical experiments to complement the theoretical analysis. We conclude with some useful notation. For any $m\geq0$ and $p\geq1$, we denote by $W^{m,p}(\Omega)$ the standard Sobolev spaces of order $m$, equipped with the norm $\|\cdot\|_{W^{m,p}(\Omega)}$ and also write $H^{m}(\Omega)$ and $H^{m}_{0}(\Omega)$ with the norm $\|\cdot\|_{H^m(\Omega)}$ when $p=2$ \cite{Adams2003Sobolev}. We denote the $L^2(\Omega)$ inner product by $(\cdot,\cdot)$. We also use Bochner spaces: For a Banach space $B$, we define by
\begin{equation*}
	H^m(0,T;B) = \{v: v(\cdot,t)\in B\ \mbox{for a.e.}\ t\in(0,T)\ \mbox{and}\ \|v\|_{H^m(0,T;B)}<\infty \}.
\end{equation*}
The space $L^\infty(0,T;B)$ is defined similarly. Throughout, we denote by $C$ a generic positive constant not necessarily the same at each occurrence but always independent of the discretization parameters $h$ and $\tau$, the noise level $\delta$ and the regularization parameter $\alpha$.

\section{Conditional stability estimates}\label{sec:stab}

In this section, we present novel conditional stability estimates for the concerned inverse problem. The analysis will also inspire the error analysis of the discrete approximations in Sections \ref{sec:err-ell} and \ref{sec:err-para}.

\subsection{Elliptic inverse problem}
We have the following conditional stability results in weighted and standard $L^2(\Omega)$ norms.
\begin{theorem}\label{thm:stab-ellipt}
Suppose that $q_1,q_2\in K\cap H^1(\Omega)$  and $f\in L^2(\Omega)$, with $\|\nabla q_1\|_{L^2(\Omega)},\|\nabla q_2\|_{L^2(\Omega)}\leq c_{q}$. Let $u(q_1)$ and $u(q_2)$ be the corresponding weak solutions of problem \eqref{equ:elliptic problem}. Then there exists a constant $C$ depends on $c_{q}$ such that
   \begin{equation*}
       \|(q_1-q_2)u(q_1)\|_{L^2(\Omega)}\leq C\|u(q_1)-u(q_2)\|_{H^1(\Omega)}^\frac12.
   \end{equation*}
   Moreover, if there exists a $\beta\geq  0$ such that
    \begin{equation}\label{eqn:positivity}
        u(q_1)(x) \geq  C \mbox{\rm dist}(x,\partial\Omega)^\beta \quad \mbox{a.e. in}\ \Omega,
    \end{equation}
    then with a constant $C$ depending on $c_q$, the following estimate holds  
    \begin{equation*}
        \|q_1-q_2\|_{L^2(\Omega)}\leq C\|u(q_1)-u(q_2)\|_{H^1(\Omega)}^{\frac{1}{2(1+2\beta)}}.
    \end{equation*}
\end{theorem}
\begin{proof}
By the weak formulations of $u(q_1)$ and $u(q_2)$, for any $\varphi\in H_0^1(\Omega)$
   \begin{equation*}
       ((q_1-q_2)u(q_1),\varphi) = -(\nabla(u(q_1)-u(q_2)),\nabla\varphi)-(q_2(u(q_1)-u(q_2)),\varphi)=:{\rm I}.
   \end{equation*}
Let $\varphi=(q_1-q_2)u(q_1)$. Note that
$\nabla\varphi=(\nabla q_1-\nabla q_2)u(q_1)+(q_1-q_2)\nabla u(q_1)$. Since $q_1\in K \cap H^1(\Omega)$, elliptic regularity theory implies $u(q_1)\in H^2(\Omega)$, and by Sobolev embedding theorem, we have $u(q_1)\in L^\infty(\Omega)$ for $d=1,2,3$. Then we have
$\|\varphi\|_{L^2(\Omega)}\leq C$ and $\|\nabla\varphi\|_{L^2(\Omega)}\leq C$, i.e., $\varphi\in H_0^1(\Omega)$. Now by the Cauchy-Schwarz inequality, we obtain the first assertion by
   \begin{equation}
       |{\rm I}|\leq C(\|\nabla(u(q_1)-u(q_2))\|_{L^2(\Omega)}+\|u(q_1)-u(q_2)\|_{L^2(\Omega)})\leq C\|u(q_1)-u(q_2)\|_{H^1(\Omega)}.
\end{equation}
Next we decompose the domain $\Omega$ into two disjoint sets $\Omega=\Omega_\rho\cup \Omega_\rho^c$, with $\Omega_\rho=\{x\in\Omega:{\rm dist}(x,\partial\Omega)\geq\rho\}$ and $\Omega_\rho^c=\Omega\setminus\Omega_\rho$, with the constant $\rho>0$ to be chosen. On the subdomain $\Omega_\rho$, we have
   \begin{align*}
     \int_{\Omega_\rho}(q_1-q_2)^2{\rm d}x &= \rho^{-2\beta}\int_{\Omega_\rho }(q_1-q_2)^2\rho^{2\beta}{\rm d}x\leq \rho^{-2\beta}\int_{\Omega_\rho }(q_1-q_2)^2\mathrm{dist}(x,\partial\Omega)^{2\beta}{\rm d}x\\
     & \leq C\rho^{-2\beta}\int_{\Omega_\rho }(q_1-q_2)^2u(q_1)^2{\rm d}x\leq C\rho^{-2\beta}\|u(q_1)-u(q_2)\|_{H^1(\Omega)}.
   \end{align*}
 Meanwhile, by the box constraint of $K$, we have
$\int_{\Omega_\rho^c}(q_1-q_2)^2{\rm d}x \leq C|\Omega_\rho^c|\leq C\rho.$
Then the desired result follows by balancing the last two estimates with $\rho$.
\end{proof}

The positivity condition \eqref{eqn:positivity} quantifies the decay rate of the solution $u(q_1)$ to zero as $\mathrm{dist}(x,\partial\Omega)\to0^+$ (due to the presence of a zero Dirichlet boundary condition). It can be verified under suitable conditions on the source $f$. This requires the following property of Green's function for the elliptic problem. The notation $B(x,r)$ denotes the ball centered at $x\in\mathbb{R}^d$ with a radius $r$.

\begin{theorem}\label{thm:Green}
Let the diffusion coefficient $a\in L^\infty(\Omega)$ with a strictly positive lower bound over $\Omega$. For any $y\in\Omega$ and $r>0$, let $G_q(x):=G_q(x,y)\in H^1(\Omega\setminus B(y,r))\cap W^{1,1}_0(\Omega)$ be Green's function for the elliptic operator $-{\rm div}(a\nabla\cdot)+qI$ (with a zero Dirichlet boundary condition). Then for $d\geq 2$, the following estimate holds 
\begin{equation*}
    G_q(x,y)\geq C|x-y|^{2-d}, \quad \mbox{for} \ |x-y|\leq \tfrac{1}{2}{\rm dist}(x,\partial\Omega).
\end{equation*}
\end{theorem}
\begin{proof}
When $q\equiv0$, the result is well known for $d\geq3$ \cite{LittmanStampacchia:1963,GruterWidman:1982}. We prove the slightly more general case for completeness. Let $\rho(x)={\rm dist}(x,\partial\Omega)$.
Since the operator $-{\rm div}(a\nabla\cdot)+qI$ is self-adjoint, there holds $G_q(x,y)=G_q(y,x)$. It suffices to prove
\begin{equation}
    G_q(x,y)\geq C|x-y|^{2-d}, \quad \mbox{for} \ |x-y|\leq \tfrac{1}{2}\rho(y).
\end{equation}
By definition, we have
\begin{equation}\label{equ: Green}
    \int_{\Omega}a(z)\nabla G_q(z,y)\cdot\nabla\varphi(z)+q(z)G_q(z,y)\varphi(z)\ \mathrm{d}z=\varphi(y), \quad\forall\varphi\in C_0^\infty(\Omega).
\end{equation}
Let $r:=|x-y|$. Consider a cut-off function $\varphi_1\in C_0^\infty(\Omega)$ with the following properties:
$\varphi_1\equiv 1$ on $ \ B(y,r)\cap (\Omega\setminus B(y,\frac r2))$ and $ \varphi_1\equiv 0$ on  $(\Omega\setminus B(y,\frac{3r}{2}))\cup B(y,\frac{r}{4})$, meanwhile $0\leq\varphi_1\leq 1$ and $|\nabla\varphi_1|\leq Cr^{-1}$. By inserting a test function $(\varphi_1(z))^2G_q(z,y)$ into \eqref{equ: Green} and applying the boundedness and uniform ellipticity of the operator and the Cauchy-Schwarz inequality, since  $G_q(z):=G_q(z,y)\in H^1(\Omega\setminus B(y,\frac{r}{4}))$, we derive
\begin{align*}
    &\int_{\frac{r}{4}\leq |z-y|\leq\frac{3r}{2}}\varphi_1(z)^2|\nabla G_q(z,y)|^2\ \mathrm{d}z  
    \leq C\int_{\frac{r}{4} |z-y|\leq\frac{3r}{2}}\varphi_1(z)G_q(z,y)|\nabla\varphi_1(z)||\nabla G_q(z,y)|\ \mathrm{d}z \\
    \leq &C\Big(\int_{\frac{r}{4}\leq |z-y|\leq\frac{3r}{2}}\varphi_1(z)^2|\nabla G_q(z,y)|^2\ \mathrm{d}z\Big)^\frac{1}{2}\Big(\int_{\frac{r}{4}\leq |z-y|\leq\frac{3r}{2}}G_q(z,y)^2|\nabla\varphi_1(z)|^2\ \mathrm{d}z\Big)^\frac{1}{2}.
\end{align*}
This and the construction of $\varphi_1$ lead to
\begin{equation}\label{inequ: gradient Green estimate on radius r}
    \int_{\frac{r}{2}\leq |z-y|\leq r}|\nabla G_q(z)|^2\mathrm{d}z\leq Cr^{d-2}\sup_{\frac{r}{4}\leq |z-y|\leq \frac{3r}{2}}|G_q(z,y)|^2.
\end{equation}
Since $G_q(x,y)\in W^{1,1}_0(\Omega)$ and $q\leq c_1$, we can choose a sufficiently small radius $r_0:=r_0(c_1)<\frac{r}{3}$ such that
\begin{equation*}
    \int_{z\in B(y,r_0)}G_q(z,y)\ \mathrm{d}z\leq\tfrac{1}{2c_1}.
\end{equation*}
Replacing the radius $r$ by $r_0$ and repeating the argument of \eqref{inequ: gradient Green estimate on radius r} yield
\begin{equation*}
    \int_{\frac{r_0}{2}\leq |z-y|\leq r_0}|\nabla G_q(z)|^2\mathrm{d}z\leq Cr_0^{d-2}\sup_{\frac{r_0}{4}\leq |z-y|\leq \frac{3r_0}{2}}|G_q(z,y)|^2.
\end{equation*}
Let $\varphi_2\in C_0^\infty(\Omega)$ be a test function of \eqref{equ: Green} satisfying 
$\varphi_2\equiv 1$ on $B(y,\frac{r_0}{2})$ and $ \varphi_2\equiv 0$ on $((\Omega\setminus B(y,r_0))\cap B(y,\frac{r}{2}))\cup (\Omega\setminus B(y,r))$, 
with $0\leq\varphi_2\leq 1$ on $\Omega$, $|\nabla\varphi_2|\leq Cr_0^{-1}$ on $B(y,r_0)$ and $|\nabla\varphi_2|\leq Cr^{-1}$ on $(\Omega\setminus B(y,r/2))\cap B(y,r)$. It follows from the boundedness of the operator, and the last three estimates that
\begin{align*}
    1 &= \int_{\Omega}a(z)\nabla G_q(z,y)\cdot\nabla\varphi_2(z)+q(z)G_q(z,y)\varphi_2(z)\ \mathrm{d}z \\
    &\leq \int_{\frac{r_0}{2}\leq|z-y|\leq r_0}a(z)\nabla G_q(z,y)\cdot\nabla\varphi_2(z)\ \mathrm{d}z+\int_{\frac{r}{2}\leq|z-y|\leq r}a(z)\nabla G_q(z,y)\cdot\nabla\varphi_2(z)\ \mathrm{d}z\\
    &\quad +\frac{1}{2}+\int_{\frac{r}{2}\leq|z-y|\leq r}q(z) G_q(z,y)\varphi_2(z)\ \mathrm{d}z \\
    &\leq Cr_0^{d-2}\sup_{\frac{r_0}{4}\leq|z-y|\leq\frac{3r_0}{2}}G_q(z,y)+
          C r^{d-2}\sup_{\frac{r}{4}\leq|z-y|\leq\frac{3r}{2}}G_q(z,y)+\frac{1}{2}+Cr^d\sup_{\frac{r}{2}\leq|z-y|\leq r}G_q(z,y) \\
    &\leq Cr^{d-2}\sup_{\frac{r_0}{4}\leq|z-y|\leq\frac{3r}{2}}G_q(z,y)+\frac{1}{2}\leq Cr^{d-2}\inf_{\frac{r_0}{4}\leq|z-y|\leq\frac{3r}{2}}G_q(z,y)+\frac{1}{2}\leq C|x-y|^{d-2}G_q(x,y)+\frac{1}{2},
\end{align*}
where we have used Harnack's inequality \cite[p. 189]{Gilbarg1977EllipticPD} for Green's function $G_q(z)$ on the compact subset $\{z\in\Omega:\frac{r_0}{4}\leq|z-y|\leq\frac{3r}{2}\}\subset\subset\Omega$, with the constant $C$ depending on the $d$, $c_1$ and $\Omega$. This completes the proof of the theorem.
\end{proof}

\begin{remark}
When $d=1$, i.e., $\Omega=(a,b)$ with $-\infty<a<b<\infty$, Green's function $G_{c_1}(x,y)$ of the operator $-\Delta+c_1I$ is explicitly given by
\begin{equation}\label{equ: 1d Green function}
G_{c_1}(x,y)=\left\{\begin{aligned}
		-\frac{e^{2\sqrt{c_1}a}(e^{2\sqrt{c_1}y}-e^{2\sqrt{c_1}b})}{2\sqrt{c_1}(e^{\sqrt{c_1}(2a+y)}-e^{\sqrt{c_1}(2b+y)})}\cdot(e^{-\sqrt{c_1}x}-e^{\sqrt{c_1}(x-2a)}), \ &\quad a\leq x\leq y, \\
		-\frac{e^{2\sqrt{c_1}b}(e^{2\sqrt{c_1}y}-e^{2\sqrt{c_1}a})}{2\sqrt{c_1}(e^{\sqrt{c_1}(2a+y)}-e^{\sqrt{c_1}(2b+y)})}\cdot(e^{-\sqrt{c_1}x}-e^{\sqrt{c_1}(x-2b)}),  \ &\quad y\leq x\leq b.
\end{aligned}\right.
\end{equation}
Now consider the asymptotics of the function $G_{c_1}(x,y)$ near the boundary. Let $y$ be close to the point $a$ and $|x-y|\leq \frac{1}{2}(y-a)$. Since $e^x-1\geq x$ on $\mathbb{R}$, we have 
\begin{equation*}
G_{c_1}(x,y)\left\{\begin{aligned}
		&\simeq e^{\sqrt{c_1}(x-2a)}-e^{-\sqrt{c_1}x}\geq C(x-a)\geq C|x-y|, \ &a\leq x\leq y, \\
		&\simeq e^{2\sqrt{c_1}y}-e^{2\sqrt{c_1}a}\geq C(y-a)\geq C|x-y|,  \ &a\leq y\leq x,
\end{aligned}\right.
\end{equation*}
where the symbol $``\simeq"$ denotes by $``="$ up to a positive constant depending on $\Omega$ and $c_1$. A similar result holds when $y$ is close to the point $b$. Since the operator $-\Delta+c_1I$ is self-adjoint, we have $$G_{c_1}(x,y)\geq C|x-y|\quad \mbox{for}\quad|x-y|\leq \tfrac{1}{2}\rho(x):={\rm dist}(x,\partial\Omega).$$
That is, the assertion in Theorem \ref{thm:Green} holds also for $d=1$. For a general potential $q\in K\cap H^1(\Omega)$, by the weak maximum principle, for any fixed $y$, we have $G_q(x,y)\geq G_{c_1}(x,y)$ a.e. $x\in \Omega$, and thus the desired assertion follows. 
\end{remark}

Now we can state a sufficient condition for the positivity condition \eqref{eqn:positivity}.
\begin{proposition}\label{propo: bete=2 in ellip}
Let $q\in H^1(\Omega)\cap K$ and $f\geq c_f>0$ a.e. in $\Omega$. Then condition \eqref{eqn:positivity} holds with $\beta=2$.
\end{proposition}
\begin{proof}
Recall that for every $y\in \Omega$, there exists a unique Green's function $G_q(\cdot,y)\in H^1(\Omega\setminus B(y,r))\cap W_0^{1,1}(\Omega)$ for the elliptic operator $-\Delta+qI$, such that 
\begin{equation*}
    \int_\Omega\nabla G_q(x,y)\cdot \nabla \varphi(x)+qG_q(x,y)\varphi(x)\ \mathrm{d}x = \varphi(y),\quad\forall \varphi\in C_0^\infty(\Omega). 
\end{equation*}
By Theorem \ref{thm:Green}, we have 
\begin{equation*}
    G_q(x,y) \geq C|x-y|^{-(d-2)}\quad \mbox{for } |x-y|\leq \tfrac12\rho(x):=\mathrm{dist}(x,\partial\Omega),\quad d\geq 2.
\end{equation*} 
Now for any $x\in \Omega$, let $B(x,\frac{\rho(x)}{2})\subset\Omega$ be the ball centered at $x$ with a radius $\frac{\rho(x)}{2}$. Since $G_q(x,y)\geq0$, $x,y\in\Omega$, we have
\begin{align*}
    u(q)(x) &= \int_\Omega G_q(x,y)f(y){\rm d} y \geq \int_{B(x,\frac{\rho(x)}{2})}f(y)G_q(x,y){\rm d}y \\
    &\geq c_fC\int_{B(x,\frac{\rho(x)}{2})}|x-y|^{-(d-2)}{\rm d}y
    \geq C\rho^2(x) = C\mathrm{dist}(x,\partial\Omega)^2,
\end{align*}
and thus the desired result follows directly.
\end{proof}

\begin{remark}
The argument for deriving the standard $L^2(\Omega)$ estimate relies heavily on the weighted $L^2(\Omega)$ estimate and the positivity condition \eqref{eqn:positivity}. Note that such an analysis strategy could also be applied to other elliptic inverse problems \cite{hao2014finite,kohn1988variational} (when equipped with alternative conditions for removing the weighted function).
\end{remark}

\subsection{Parabolic inverse problem}

The next result gives a conditional stability estimate for the parabolic inverse problem.
\begin{theorem}
Suppose that $q_1,q_2\in K\cap H^1(\Omega)$ with $\|\nabla q_1\|_{L^2(\Omega)},\|\nabla q_2\|_{L^2(\Omega)}\leq c_q$, $u_0\in H^2(\Omega)\cap H_0^1(\Omega)$ and $f\in H^1(0,T;L^2(\Omega))$. Let $u(q_1)$ and $u(q_2)$ be the corresponding weak solutions of problem \eqref{equ:parabolic problem}. Then with $T_0\leq s\leq t\leq T$, there exists a constant $C$ depending on $c_q$ such that
\begin{equation*}
\int_{T_0}^T\int_{T_0}^t\int_{s}^t\|(q_1-q_2)u(q_1)(\xi)\|_{L^2(\Omega)}\ \mathrm{d}\xi\mathrm{d}s\mathrm{d}t\leq C\|u(q_1)-u(q_2)\|_{L^2(T_0,T;H^1(\Omega))}^\frac12.
\end{equation*}
Moreover, if there exists a $\beta\geq 0$ such that 
\begin{equation}\label{eqn:positivity-parab}
u(q_1)(x,t)\ge C\mathrm{dist}(x,\partial\Omega)^\beta \quad\mbox{a.e. in}\ \Omega
\end{equation}
for any $t\in(T_0,T)$, then there exists a constant $C$ depending on $c_q$ such that
\begin{equation*}
        \|q_1-q_2\|_{L^2(\Omega)}\leq C\|u(q_1)-u(q_2)\|_{L^2(T_0,T;H^1(\Omega))}^{\frac{1}{2(1+2\beta)}}.
    \end{equation*}
\end{theorem}
\begin{proof}
By the weak formulations of $u(q_1)$ and $u(q_2)$, for any $\varphi\in H_0^1(\Omega)$
\begin{equation*}
    ((q_1-q_2)u(q_1),\varphi) = -(\partial_{\xi} u(q_1)-\partial_{\xi}u(q_2),\varphi) -(\nabla(u(q_1)-u(q_2)),\nabla\varphi)-(q_2(u(q_1)-u(q_2)),\varphi)=:{\rm I_1}+{\rm I_2}+{\rm I_3}.
    \end{equation*}
Let $\varphi=(q_1-q_2)u(q_1)$. Then $\nabla\varphi=(\nabla q_1-\nabla q_2)u(q_1)+(q_1-q_2)\nabla u(q_1)$.  By the standard parabolic regularity theory \cite{Evans2010PartialDE}, problem \eqref{equ:variational problem in parabolic system} has a unique solution $u(q_1)\in H^1(0,T;H_0^1(\Omega))\cap L^\infty(0,T;H^2(\Omega)) $, and then by Sobolev embedding theorem \cite{Adams2003Sobolev}, $u(q_1)\in L^\infty(0,T;L^\infty(\Omega))$. Then there holds $\|\varphi\|_{L^2(\Omega)}\leq C$ and $\|\nabla\varphi\|_{L^2(\Omega)}\leq C$. Thus we have $\varphi\in H_0^1(\Omega)$. Meanwhile, the Cauchy-Schwarz inequality and the box constraint yield
\begin{equation*}
\int_{T_0}^T|{\rm I_2}|\ \mathrm{d}t\leq C\|\nabla(u(q_1)-u(q_2))\|_{L^2(T_0,T;L^2(\Omega))}\quad\mbox{and}\quad
\int_{T_0}^T|{\rm I_3}| \mathrm{d}t\leq C\|u(q_1)-u(q_2)\|_{L^2(T_0,T;L^2(\Omega))}.
\end{equation*}
It remains to bound the term ${\rm I_1}$. By integration by parts, we have
\begin{align*}
    \int_{s}^t(\partial_\xi u(q_1)(\xi)-\partial_\xi u(q_2)(\xi),\varphi(\xi))\ \mathrm{d}\xi&=(u(q_1)(t)-u(q_2)(t),\varphi(t))-(u(q_1)(s)-u(q_2)(s),\varphi(s))\\
     &\quad-\int_{s}^t(u(q_1)(\xi)-u(q_2)(\xi),\partial_\xi\varphi(\xi))\ \mathrm{d}\xi.
\end{align*}
For the first two terms, by the Cauchy-Schwarz inequality, since $\|\varphi(t)\|_{L^2(\Omega)}\leq C$, we have 
\begin{align*}
    & \Big|\int_{T_0}^T\int_{T_0}^t(u(q_1)-u(q_2),\varphi)(t)-(u(q_1)-u(q_2),\varphi)(s)\ \mathrm{d}s\mathrm{d}t\Big| 
   \leq C\|u(q_1)-u(q_2)\|_{L^2(T_0,T;L^2(\Omega))}.
\end{align*}
Next by the regularity $u\in H^1(0,T;L^2(\Omega))$, and the box constraint in $K$, we have $\partial_\xi\varphi = (q_1-q_2)\partial_{\xi}u(q_1)(\xi)\in L^2(0,T;L^2(\Omega))$. Using  Cauchy-Schwarz inequality leads to
    \begin{align*}
    \Big|\int_{T_0}^T\int_{T_0}^t\int_s^t(u(q_1)(\xi)-u(q_2)(\xi),\partial_{\xi}\varphi(\xi))\ \mathrm{d}\xi\mathrm{d}s\mathrm{d}t\Big|
    &\leq C\Big(\int_{T_0}^T\|u(q_1)-u(q_2)\|^2_{L^2(\Omega)}\ \mathrm{d}t\Big)^{\frac{1}{2}}\Big(\int_{T_0}^T\|\partial_t\varphi\|^2_{L^2(\Omega)}\ \mathrm{d}t\Big)^{\frac{1}{2}} \\
    &\leq C\|u(q_1)-u(q_2)\|_{L^2(T_0,T;L^2(\Omega))}\|\partial_tu\|_{L^2(T_0,T;L^2(\Omega))}.
\end{align*}
These estimates directly imply the first desired assertion. Under the positivity condition \eqref{eqn:positivity-parab}, the $L^2(\Omega)$ estimate follows from the argument of Theorem \ref{thm:stab-ellipt}.
\end{proof}

The next result gives a sufficient condition on the positivity condition \eqref{eqn:positivity-parab}.
\begin{proposition}\label{propo: beta = 2 in para system}
Let $q\in H^1(\Omega)\cap K$, the source $f\geq c_f>0$ and $\partial_tf\leq 0$ a.e. in $\Omega\times(0,T)$, $u_0\geq 0$ and $f(0)+\Delta u_0-qu_0\leq 0$ a.e. in $\Omega$. Then there exists a positive constant $C$, depending only on $c_0$, $c_1$, $c_f$ and $\Omega$, such that the positivity condition \eqref{eqn:positivity-parab} holds with $\beta=2$.
\end{proposition}
\begin{proof}
Since $f\geq0$ and $u_0\geq 0$, the standard parabolic maximum principle (see, e.g., \cite{Nirenberg:1953,Friedman:1958}) implies $u\geq 0$, a.e. in $\Omega\times[0,T].$
Let $w:=\partial_tu$, which satisfies \begin{equation*}
\left\{\begin{aligned}
	\partial_tw-\Delta w +qw &= \partial_tf, \ &\mbox{in}&\ \Omega\times(0,T), \\
	w&=0, \ &\mbox{on}&\ \partial\Omega\times(0,T), \\
	w(0)&=f(0)+\Delta u_0-q u_0, \ &\mbox{in}&\ \Omega.
\end{aligned}\right.
\end{equation*}
Since $\partial_tf\leq 0$ a.e. in $\Omega\times(0,T)$ and $f(0)+\Delta u_0-qu_0\leq 0$ a.e. in $\Omega$, the parabolic maximum principle yields $w\leq 0$ a.e. in $\Omega\times[0,T]$. It suffices to prove that \eqref{eqn:positivity-parab} with $\beta=2$ holds for any $t\in(0,T]$. By fixing $t\in [T_0,T]$, we have $f(t)-w(t)\in L^2(\Omega)$. Then consider the following elliptic problem
\begin{equation*}
    \left\{\begin{aligned}
		-\Delta u(t) +qu(t) &= f(t)-w(t), \ &\mbox{in}&\ \Omega, \\
		u(t)&=0, \ &\mbox{on}&\ \partial\Omega.
\end{aligned}\right.
\end{equation*}
By the property of Green's function $G_{q}(x,y)$ in Theorem \ref{thm:Green}, there holds 
\begin{align*}
    u(q)(x,t) &= \int_\Omega G_{q}(x,y)(f(y,t)-w(y,t)) {\rm d}y \geq \int_{B(x,\frac{\rho(x)}{2})}f(y,t)G_{q}(x,y) {\rm d} y \\
    &\geq c_fC\int_{B(x,\frac{\rho(x)}{2})}|x-y|^{-(d-2)} {\rm d}y
    \geq C\rho^2(x) = C\mathrm{dist}(x,\partial\Omega)^2,
\end{align*}
for any $x\in\Omega$ and $t\in[T_0,T]$, i.e., the positivity condition \eqref{eqn:positivity-parab} with $\beta=2$ holds.
\end{proof}

\section{Error analysis for the elliptic inverse problem}
\label{sec:err-ell}

Now we formulate the regularized output least-squares formulation for the elliptic inverse problem, discretize the continuous formulation by the Galerkin FEM with continuous piecewise linear elements, and provide a complete error analysis. 

\subsection{Regularization problem and its FEM approximation}
\label{subsec:Tikhonov regularization problem and its FEM approximation in elliptic system}
To reconstruct the coefficient $q$, we
employ the standard Tikhonov regularization with an $H^1(\Omega)$ seminorm penalty \cite{Engl1996Regularization,ItoJin:2015}, minimizes the following regularized functional:
\begin{equation}\label{equ:Tikhonov problem in elliptic system}
	\min_{q\in K} J_{\alpha}(q)=\frac{1}{2}\|u(q)-z^{\delta}\|^2_{L^2(\Omega)}+\frac{\alpha}{2}\|\nabla q\|_{L^2(\Omega)}^2,
\end{equation}
where $u(q)\in H^1_0(\Omega)$ satisfies
\begin{equation}\label{equ:variational problem in elliptic system}	
   (\nabla u(q),\nabla\varphi)+(qu(q),\varphi)=(f,\varphi), \quad\forall \varphi\in H^1_0(\Omega).
\end{equation}
Recall that the given data $z^\delta\in L^2(\Omega)$ is noisy with a noise level $\delta$ relative to the exact data $u(q^\dag)$ (corresponding to the exact radiativity $q^\dag$), i.e.,
$\|u(q^\dag)-z^\delta\|_{L^2(\Omega)}= \delta.$ The continuous problem \eqref{equ:Tikhonov problem in elliptic system}--\eqref{equ:variational problem in elliptic system} 
is well-posed in the sense that it has at least one global minimizer, and the minimizer is continuous with respect to the perturbations in the data, and further as the noise level tends to zero, the sequence of minimizers converges to the exact solution in $H^1(\Omega)$ (if $\alpha$ is chosen properly) \cite{EnglKunischNeubauer:1989,Engl1996Regularization,ItoJin:2015}.

To discretize problem \eqref{equ:Tikhonov problem in elliptic system}--\eqref{equ:variational problem in elliptic system}, we employ the standard Galerkin FEM \cite{Brenner2002The}. Let $\mathcal{T}_h$ be a shape regular quasi-uniform simplicial triangulation of the domain $\Omega$, with a grid size $h$. On the triangulation $\mathcal{T}_h$, we define the conforming piecewise linear finite element spaces $V_h$ and $V_{h0}$ by
\begin{equation*}
	V_{h}:=
	\{v_{h}\in H^1(\Omega):	v_{h}|_{T}\text{ is a linear polynomial},\, \forall T\in\mathcal{T}_h
	\}
\end{equation*}
and $V_{h0}:=V_h\cap H_0^1(\Omega)$. We use the spaces $V_{h0}$ and $V_h$ to approximate the state $u$ and the parameter $q$, respectively. The following inverse inequality holds in the finite element space $V_{h0}$ \cite{Brenner2002The}
\begin{equation}\label{inequ:inverse inequality}
	\|v_h\|_{H^1(\Omega)}\leq  Ch^{-1}\|v_h\|_{L^2(\Omega)},\quad \forall v_h\in V_{h0}.
\end{equation}
We denote by $P_h$ the standard $L^2(\Omega)$-projection operator associated with the finite element space $V_{h0}$. Then it is known that for $s=1,2$ \cite{Ciarlet1991BasicEE,Thome2006GalerkinFE}:
\begin{align}
\label{inequ:P_h}
\|v-P_hv\|_{L^2(\Omega)}+h\|\nabla(v-P_hv)\|_{L^2(\Omega)}&\leq  Ch^s\|v\|_{H^s(\Omega)}, \quad \forall v\in H^s(\Omega)\cap H^1_0(\Omega).
\end{align}
Let $\Pi_h$ be the Lagrange interpolation operator associated with the finite element space $V_h$. It satisfies the following error estimate for $s=1,2$ and $1 \le p\le \infty$ (with $sp>d$):
\begin{align}
  \|v-\Pi_hv\|_{L^p(\Omega)} + h\|v-\Pi_hv\|_{W^{1,p}(\Omega)} & \leq ch^s \|v\|_{W^{s,p}(\Omega)}, \quad \forall v\in W^{s,p}(\Omega). \label{inequ:pi_h on H^2}
\end{align}

Now we can state the finite element approximation of problem \eqref{equ:Tikhonov problem in elliptic system}-\eqref{equ:variational problem in elliptic system}:
\begin{equation}\label{equ:discrete minimizing problem in elliptic system}
	\min_{q_h\in K_h} J_{\alpha,h}(q_h)=\frac{1}{2}\|u_h(q_h)-z^{\delta}\|^2_{L^2(\Omega)}+\frac{\alpha}{2}\|\nabla q_h\|_{L^2(\Omega)}^2,
\end{equation}
with $K_h=K \cap V_h$, and $u_h(q_h)\in V_{h0}$ satisfying
\begin{equation}\label{equ:finite element problem in elliptic system}	
	(\nabla u_h(q_h),\nabla\varphi_h)+(q_hu_h(q_h),\varphi_h)=(f,\varphi_h), \quad\forall \varphi_h\in V_{h0}.
\end{equation}
The discrete problem is well-posed: there exists at least one global minimizer $q_h^*\in K_h$ to problem \eqref{equ:discrete minimizing problem in elliptic system}--\eqref{equ:finite element problem in elliptic system}, and it depends continuously on the data perturbation. The main objective of this work is to bound the error $q_h^*-q^\dag$ of the approximation $q_h^*$.

\subsection{Error estimates}
Now we derive a weighted $L^2(\Omega)$ estimate of the error $q^\dag-q_h^*$, under the following assumption. By the standard elliptic regularity theory \cite{Evans2010PartialDE} and Sobolev embedding theorem  \cite{Adams2003Sobolev}, Assumption \ref{assum:Assumption in elliptic system} implies $u(q^{\dagger})\in H^2(\Omega)\hookrightarrow L^\infty(\Omega)$, for $d=1,2,3$.

\begin{assumption}\label{assum:Assumption in elliptic system}
$q^{\dagger}\in H^2(\Omega)\cap K$ and $f\in L^2(\Omega)$.
\end{assumption}


Next we give the main result of this section, i.e., a novel weighted $L^2(\Omega)$ error estimate.
\begin{theorem}\label{thm:error estimates of q-q_h with weight}
Let Assumption \ref{assum:Assumption in elliptic system} be fulfilled. Let $q^{\dagger}\in K$ be the exact potential, $u(q^{\dagger})$ be the solution of problem \eqref{equ:variational problem in elliptic system}, and $q_h^*\in K_h$ be a minimizer of problem \eqref{equ:discrete minimizing problem in elliptic system}-\eqref{equ:finite element problem in elliptic system}. Then with $\eta=h^2+\delta+\alpha^{\frac{1}{2}}$, there holds
	\begin{equation}\label{inequ:error estimates of q-q_h}
		\|(q^{\dagger}-q_h^*)u(q^\dagger)\|_{L^2(\Omega)}\leq  C(h^{\frac{1}{2}}+\alpha^{\frac{1}{4}}+\min(h^\frac12+h^{-\frac{1}{2}}\eta^{\frac{1}{2}},1))\alpha^{-\frac{1}{4}}\eta^{\frac{1}{2}},
	\end{equation}
	where the constant $C$ depends only on $q^\dagger$.
\end{theorem}

The proof employs two preliminary estimates.
\begin{lemma}\label{lem:error estimates of u(q)-u(pi_hq)}
	Let Assumption \ref{assum:Assumption in elliptic system} be fulfilled. Then there holds
	\begin{equation}\label{inequ:error estimates of u(q)-u(pi_hq) in H^1}
		\|u(q^{\dagger})-u_h(\Pi_hq^{\dagger})\|_{L^2(\Omega)} + h \|\nabla (u(q^{\dagger})-u_h(\Pi_hq^{\dagger}))\|_{L^2(\Omega)}\leq  Ch^2.
	\end{equation}
\end{lemma}
\begin{proof}
C\'{e}a's Lemma and the standard duality argument imply
\begin{equation}\label{eqn:esti-ellip-1}
	\|u(q^{\dagger})-u_h(q^{\dagger})\|_{L^2(\Omega)} + h \|\nabla(u(q^{\dagger})-u_h(q^{\dagger}))\|_{L^2(\Omega)}  \le Ch^2 \|  u(q^\dagger) \|_{H^2\II}\leq Ch^2.
	\end{equation}
Then it suffices to bound $w_h = u_h(\Pi_h q^{\dagger})-u_h(q^{\dagger})$. Clearly, $w_h$ satisfies for any $\varphi_h \in V_{h0}$
\begin{align*}
 & (\nabla w_h, \nabla \varphi_h ) + ((\Pi_h q^\dagger) w_h , \varphi_h) = (( q^{\dagger}-\Pi_h q^{\dagger} ) u_h(q^{\dagger}) , \varphi_h) \\
  = & (( q^{\dagger}-\Pi_h q^{\dagger} ) u(q^{\dagger}) , \varphi_h) + (( q^{\dagger}-\Pi_h q^{\dagger} ) (u(q^{\dagger}) - u_h(q^\dagger)), \varphi_h)
 =:   {\rm I}.
 \end{align*}
 Letting $\varphi_h = w_h$, the estimate \eqref{eqn:esti-ellip-1}, and the approximation property of $\Pi_h$ in  \eqref{inequ:pi_h on H^2} give
 \begin{equation*} 
|{\rm I}| \le C h^2 \big( \| q \|_{H^2\II} \|  u(q) \|_{L^\infty\II} \|  w_h \|_{L^2\II} +  \| q \|_{L^\infty \II} \|  u(q) \|_{H^2\II} \|  w_h \|_{L^2\II} \big).
 \end{equation*}
Consequently, we have
\begin{equation*} 
 \|  \nabla w_h  \|_{L^2\II}^2 + c_0   \|   w_h \|_{L^2\II}^2 \le C h^2 \|  w_h \|_{L^2\II}.
  \end{equation*} 
Since $c_0\geq 0$, this and Poincare's inequality lead to 
 \begin{equation*}
 \|  \nabla w_h  \|_{L^2\II}  \le C h^2.
\end{equation*}
This, Poincar\'{e}'s inequality, \eqref{eqn:esti-ellip-1} and the triangle inequality imply the assertion.
\end{proof}

The next result gives a crucial a priori bound on $\|\nabla q_h^*\|_{L^2(\Omega)}$ and state the approximation error $\|u(q^\dag)-u_h(q_h^*)\|_{L^2(\Omega)}$. Note that this estimate allows bounding $\|\nabla q_h^*\|_{L^2(\Omega)}$ a priori by $C\alpha^{-\frac12}(h^2+\delta+\alpha^\frac12)$, where the constant $C$ depends only on the a priori regularity of the exact potential $q^\dag$. This estimate shows explicitly the delicate interplay of the parameters $h$, $\alpha$ and $\delta$ and it will play a central role in the error analysis below.
\begin{lemma}\label{lemma:error estimate of u(q)-u_h(q_h^*) and grad q_h^* in L^2}
     Let the assumption in Theorem \ref{thm:error estimates of q-q_h with weight} be fulfilled. Then the following estimate holds
     \begin{equation}\label{inequ:error estimate of u(q)-u_h(q_h^*) and grad q_h^* in L^2}
     	\|u(q^\dagger)-u_h(q_h^*)\|_{L^2(\Omega)}+\alpha^{\frac{1}{2}}\|\nabla q_h^*\|_{L^2(\Omega)}\leq C(h^2+\delta+\alpha^{\frac{1}{2}}).
     \end{equation}
\end{lemma}
\begin{proof}
Since $q_h^*$ is the minimizer of the system \eqref{equ:discrete minimizing problem in elliptic system}-\eqref{equ:finite element problem in elliptic system} and $\Pi_hq^\dagger\in K_h$, we have $$J_{\alpha,h}(q_h^*)\leq J_{\alpha,h}(\Pi_hq^\dagger).$$ 
Then the definition of $\delta$ and Lemma \ref{lem:error estimates of u(q)-u(pi_hq)} imply
    \begin{align*}
		\|u_h(q_h^*)-z^\delta\|_{L^2(\Omega)}^2&+\alpha\|\nabla q_h^*\|_{L^2(\Omega)}^2\leq \|u_h(\Pi_hq^\dagger)-z^\delta\|_{L^2(\Omega)}^2
		+\alpha\|\nabla \Pi_hq^\dagger\|_{L^2(\Omega)}^2\\
		\leq & C(\|u_h(\Pi_hq^\dagger)-u(q^\dagger)\|_{L^2(\Omega)}^2+\|u(q^\dagger)-z^\delta\|_{L^2(\Omega)}^2+\alpha \|\nabla \Pi_hq^\dagger\|_{L^2(\Omega)}^2)\\
		\leq &C(h^4+\delta^2+\alpha),
	\end{align*}
where the last line follows from the inequality \eqref{inequ:pi_h on H^2} and the inverse inequality \eqref{inequ:inverse inequality}, i.e., $\|\nabla\Pi_hq^\dag\|_{L^2(\Omega)}\leq C\| q^\dag\|_{H^1(\Omega)}$. Then by the triangle inequality, we deduce
   \begin{align*}
   	\|u(q^\dagger)-u_h(q_h^*)\|^2_{L^2(\Omega)}+\alpha\|\nabla q_h^*\|^2_{L^2(\Omega)}&\leq  C(\|u(q^\dagger)-z^\delta\|_{L^2(\Omega)}^2+\|u_h(q_h^*)-z^\delta\|_{L^2(\Omega)}^2+\alpha\|\nabla q_h^*\|_{L^2(\Omega)}^2)\\
   	&\leq  C(h^4+\delta^2+\alpha).
   \end{align*} 
This completes the proof of the lemma. 
\end{proof}

Now we can prove Theorem \ref{thm:error estimates of q-q_h with weight}, which relies heavily on the novel test function $\varphi=(q^\dag-q_h^*)u(q^\dag)$, and the overall strategy is inspired by the conditional stability analysis in Section \ref{sec:stab}.
\begin{proof}
For any $\varphi\in H_0^1(\Omega)$, it follows from the weak formulations of $u(q^\dag)$ and $u_h(q_h^*)$ that
\begin{equation*}
\begin{aligned}
	((q^\dagger-q_h^*)u(q^\dagger),\varphi)
	&=((q^\dagger-q_h^*)u(q^\dagger),\varphi-P_h\varphi)+((q^\dagger-q_h^*)u(q^\dagger),P_h\varphi)\\
	&=((q^\dagger-q_h^*)u(q^\dagger),\varphi-P_h\varphi)+(\nabla(u_h(q_h^*)-u(q^\dagger)),\nabla P_h\varphi)+(q_h^*(u_h(q_h^*)-u(q^\dagger)),P_h\varphi)\\
	&=: {\rm I}_1+{\rm I}_2+{\rm I}_3.
\end{aligned}
\end{equation*}
Let $\varphi=(q^\dagger-q_h^*)u(q^\dagger)$. Then direct computation gives
$\nabla\varphi = u(q^\dagger)(\nabla q^\dagger -\nabla q_h^*)+(q^\dagger-q^*_h)\nabla u(q^\dagger).$
By Assumption \ref{assum:Assumption in elliptic system} and the a priori regularity $u(q^\dag)\in H^2(\Omega)\hookrightarrow L^\infty(\Omega)$ for $d=1,2,3$, we have 
$$\|  \varphi \|_{L^2\II} \le C\quad\text{and}\quad \|\nabla\varphi\|_{L^2(\Omega)}\leq  C(1+\|\nabla q_h^*\|_{L^2(\Omega)}).$$
Thus, $\varphi\in H_0^1(\Omega)$. We bound the three terms. 
By Lemma \ref{lemma:error estimate of u(q)-u_h(q_h^*) and grad q_h^* in L^2} and \eqref{inequ:P_h}, we bound the term ${\rm I}_1$ by
\begin{equation*}
	|{\rm I}_1|\leq  Ch\|\nabla\varphi\|_{L^2(\Omega)}\leq  Ch(1+\|\nabla q_h^*\|_{L^2(\Omega)})\leq  Ch(1+\alpha^{-\frac{1}{2}}\eta)\leq  Ch\alpha^{-\frac{1}{2}}\eta,
\end{equation*}
where the constant $C$ depends on $q^\dagger$. By applying Lemma  \ref{lemma:error estimate of u(q)-u_h(q_h^*) and grad q_h^* in L^2} and \eqref{inequ:P_h} again and also the inverse inequality \eqref{inequ:inverse inequality}, the term ${\rm I}_2$ can be bounded by
\begin{equation*}
\begin{aligned}
	|{\rm I}_2|&\leq \|\nabla(u(q^\dagger)-u_h(q_h^*))\|_{L^2(\Omega)}\|\nabla\varphi\|_{L^2(\Omega)} \\
	&\leq  (\|\nabla(u(q^\dagger)- P_hu(q^\dagger))\|_{L^2(\Omega)}+Ch^{-1}\|P_hu(q^\dagger)-u_h(q_h^*)\|_{L^2(\Omega)})\|\nabla\varphi\|_{L^2(\Omega)} \\
	&\leq  C(h+h^{-1}\|u(q^\dagger)-u_h(q_h^*)\|_{L^2(\Omega)})\|\nabla\varphi\|_{L^2(\Omega)}  \\
	&\leq  C(h+h^{-1}\eta)(1+\|\nabla q_h^*\|_{L^2(\Omega)})\leq  C(h+h^{-1}\eta)\alpha^{-\frac{1}{2}}\eta. 
\end{aligned}
\end{equation*}
Meanwhile, by the a priori estimate, we have 
$$\|\nabla(u(q^\dag)-u_h(q_h^*))\|_{L^2(\Omega)}\leq C.$$ 
Thus, we obtain
$$|{\rm I}_2| \leq  C\min(h+h^{-1}\eta,1)\alpha^{-\frac{1}{2}}\eta.$$

Finally, the bound on the term ${\rm I}_3$ follows from Lemma \ref{lemma:error estimate of u(q)-u_h(q_h^*) and grad q_h^* in L^2} and the $L^2(\Omega)$ stability of $P_h$ by
\begin{equation*}
	|{\rm I}_3|\leq  C \| P_h\varphi  \|_{L^2(\Omega)} \|u(q^\dagger)-u_h(q_h^*)\|_{L^2(\Omega)}\leq  C\eta.
\end{equation*}
Then the desired estimate follows from the bounds on ${\rm I}_i$.
\end{proof}

From Theorem \ref{thm:error estimates of q-q_h with weight}, we can derive two $L^2(\Omega)$ estimates on the error $q^\dagger-q_h^*$. First, we give an interior $L^2$-error estimate, by means of maximum principle.
\begin{corollary}\label{coro:interior error estimates in $L^2$-norm of elliptic problem}
Let the assumptions in Theorem \ref{thm:error estimates of q-q_h with weight} be fulfilled and the source $f\not\equiv0$ be nonnegative a.e. in $\Omega$. Then for any compact subset $\Omega'\subset\subset\Omega$ with {\rm dist}$(\overline{\Omega'},\partial\Omega)>0$, there exists a positive constant $C$, depending on {\rm dist}$(\overline{\Omega'},\partial\Omega)$ and $q^\dagger$, such that
\begin{equation*}
	\|q^\dagger-q^*_h\|_{L^2(\Omega')}\leq  C(h^{\frac{1}{2}}+\alpha^{\frac{1}{4}}+\min(h^\frac12+ h^{-\frac{1}{2}}\eta^{\frac{1}{2}},1))\alpha^{-\frac{1}{4}}\eta^{\frac{1}{2}}.
\end{equation*}
\end{corollary}
\begin{proof}
Let $w\in H^2(\Omega)\cap H^1_0(\Omega)$ be the solution of problem \eqref{equ:variational problem in elliptic system} with $q$ replaced by $c_1$. Since $f\geq 0$ a.e. in  $\Omega$, by the maximum principle \cite{Brezis2010FunctionalAS}, $u(q^\dag)\geq0$. Note that $u(q^\dag)-w$ satisfies
\begin{equation*}
    (\nabla (u(q^\dag)-w),\nabla \varphi) + (q^\dag(u(q^\dag)-w),\varphi) = ((c_1-q^\dag)u(q^\dag),\varphi),\quad \forall \varphi\in H_0^1(\Omega).
\end{equation*}
By the maximum principle \cite{Brezis2010FunctionalAS}, $u(q^\dag)-w\geq 0$ a.e. in $\Omega$.
Meanwhile, by Sobolev embedding theorem \cite{Adams2003Sobolev}, $H^2(\Omega)\hookrightarrow C(\overline{\Omega})$, we have $u(q^\dagger)\geq  w\geq  0 \ \mbox{in}\ \Omega$. Then by Theorem \ref{thm:error estimates of q-q_h with weight}
\begin{equation*}
 \|(q^\dagger-q^*_h)w\|_{L^2(\Omega)}\leq  C(h^{\frac{1}{2}}+\alpha^{\frac{1}{4}}+\min(h^\frac12+h^{-\frac{1}{2}}\eta^{\frac{1}{2}},1))\alpha^{-\frac{1}{4}}\eta^{\frac{1}{2}}.
\end{equation*}
Note that for any $\Omega'\subset\subset \Omega$, by \cite[Theorem 1]{Vzquez1984ASM} and $w\in C(\overline{\Omega})$, there exists a positive constant $C$ depending on {\rm dist}$(\overline{\Omega'},\partial\Omega)$ such that $w\geq  C>0$ in $\Omega'$. The desired estimate follows directly.
\end{proof}

The next result gives an $L^2(\Omega)$ error estimate under the positivity condition \eqref{eqn:positivity}.
\begin{corollary}\label{coro:standard L2 in ellip}
Let the conditions in Theorem \ref{thm:error estimates of q-q_h with weight} be fulfilled, and condition \eqref{eqn:positivity} holds. Then there holds 
    \begin{equation*}
      \|q^\dagger-q_h^*\|_{L^2(\Omega)}\leq C ((h^{\frac{1}{2}}+\alpha^{\frac{1}{4}}+\min(h^\frac12+h^{-\frac{1}{2}}\eta^{\frac{1}{2}},1))\alpha^{-\frac{1}{4}}\eta^{\frac{1}{2}})^\frac{1}{1+2\beta}.
    \end{equation*}
\end{corollary}
\begin{proof}
The proof is identical with that in Theorem \ref{thm:stab-ellipt}, and hence omitted.
\end{proof}

\begin{remark}\label{remark:index match in elliptic system}
The error estimates provide useful guidelines for choosing the algorithmic parameters, e.g., $\alpha$ and $h$, in order to achieve the best possible convergence rates of the discrete approximations $q_h^*$ in terms of the noise level $\delta$. Indeed, by properly balancing the terms in the upper bound, we should choose $\alpha\sim\delta^2$ and $h\sim\delta^{\frac{1}{2}}$ so as to obtain the following interior estimates:  $\|q^{\dagger}-q_h^*\|_{L^2(\Omega')}\leq  C\delta^{\frac{1}{4}}.$
Similarly, there holds
$\| q_h - q^\dagger  \|_{L^2(\Omega)}  \le C \delta^{\frac{1}{4(1+2\beta)}},$
under the positivity condition \eqref{eqn:positivity}. Note that the choice $\alpha\sim \delta^2$ contrasts sharply with that in standard regularization theory \cite{Engl1996Regularization,ItoJin:2015} which typically assumes a slower decay than $\delta^2$, but it is actually the most common choice when using conditional stability estimates \cite{ChengYamamoto:2000,EggerHofmann:2018}. It is instructive to compare the rate with the conditional stability estimate in Theorem \ref{thm:stab-ellipt}. Specifically, by the a prior regularity estimate $u(q)\in H^2(\Omega)$ for $q\in K\cap H^1(\Omega)$ and the
Gagliardo-Nirenberg interpolation inequality 
$$ \|v\|_{H^1(\Omega)}\leq C\|v\|^{\frac{1}{2}}_{L^2(\Omega)}\|v\|^{\frac{1}{2}}_{H^2(\Omega)},$$ 
we have
\begin{equation*}
\|q_1-q_2\|_{L^2(\Omega)}\leq C\|u(q_1)-u(q_2)\|_{L^2(\Omega)}^\frac{1}{4(1+2\beta)},\quad \forall q_1,q_2\in K\cap H^1(\Omega).    
\end{equation*} 
Thus the error estimate in Corollary \ref{coro:standard L2 in ellip} is consistent with the conditional stability estimate in Theorem \ref{thm:stab-ellipt}.
\end{remark}

\section{Error analysis for the parabolic inverse problem}
\label{sec:err-para}

Now we turn to the convergence analysis for parabolic systems.
\subsection{Regularized formulation and its FEM approximation}
\label{subsec:Tikhonov regularization problem and its FEM approximation in parabolic system}

Like the elliptic case in Section \ref{sec:err-ell}, we consider the inverse problem of recovering a space-dependent potential $q$ from the following noisy distributed observation $z^\delta$
\begin{equation*}
    z^\delta(x,t) = u(q^\dag)(x,t) + \xi(x,t),\quad \mbox{in }\Omega\times(T_0,T), 
\end{equation*}
where $0\leq T_0<T$, and the accuracy of $z^\delta$ is measured by the noise level $\delta$, defined by 
$\delta = \|u(q) -z^\delta \|_{L^2(T_0,T;L^2(\Omega))}.$
To reconstruct the potential $q$ from the data $z^\delta$, we employ the standard Tikhonov regularization, which minimizes the following regularized functional:
\begin{equation}\label{equ:Tikhonov problem in parabolic system}
\min_{q\in K} J_{\alpha}(q)=\frac{1}{2}\|u(q)(t)-z^{\delta}(t)\|^2_{L^2(T_0,T;L^2(\Omega))}+\frac{\alpha}{2}\|\nabla q\|_{L^2(\Omega)}^2,
\end{equation}
where $u(t)\equiv u(q)(t)\in H^1_0(\Omega)$ with $u(0)=u_0$ satisfies 
\begin{equation}\label{equ:variational problem in parabolic system}	
	(\partial_tu(t),\varphi)+(\nabla u(t),\nabla\varphi)+(qu(t),\varphi)=(f,\varphi), \quad\forall \varphi\in H^1_0(\Omega),\ \mbox{a.e.}\ t\in(0,T).
\end{equation}
Similar to the elliptic case, the well-posedness of the continuous formulation \eqref{equ:Tikhonov problem in parabolic system}--\eqref{equ:variational problem in parabolic system} can be shown easily using the direct method in calculus of variation \cite{YamamotoZou:2001}.

For the spatial discretization of problem \eqref{equ:Tikhonov problem in parabolic system}--\eqref{equ:variational problem in parabolic system}, we apply the Galerkin FEM described in Section \ref{sec:err-ell}, i.e., to approximate the state variable $u\in H^1_0(\Omega)$ by $V_{h0}$ and the unknown potential $q$ by the space $V_h$. For the time discretization, we employ the backward Euler method on a uniform time grid. We divide the time interval $(0,T)$ into $N$ equal subintervals with a time step size $\tau$ and the time grids $t^n=n\tau$, $n=0,\ldots,N$. Further, we denote by $v^n=v(t^n)$ and define the backward difference quotient $\partial_\tau v^n$ and and the piecewise constant $L^2$-projection in the cell $(t^{n-1},t^n)$ by
$$\partial_\tau v^n:=\tau^{-1}(v^n-v^{n-1})\quad  \mbox{and}\quad \bar{v}^n:=\tau^{-1}\int_{t^{n-1}}^{t^n}v(t)\ \mathrm{d}t,$$ 
respectively. Throughout, we assume that $T_0$ is at a grid point, with $N_0=T_0/\tau$ for some $N_0\in\mathbb{N}$ (which depends on the step size $\tau$).

Then the fully discrete problem for problem \eqref{equ:Tikhonov problem in parabolic system}-\eqref{equ:variational problem in parabolic system} reads:
\begin{equation}\label{equ:discrete minimizing problem in parabolic system}
	\min_{q_h\in K_h} J_{\alpha,h,\tau}(q_h)=\frac{\tau}{2}\sum_{n=N_{0}}^{N}\|u^n_h(q_h)-z_n^{\delta}\|^2_{L^2(\Omega)}+\frac{\alpha}{2}\|\nabla q_h\|_{L^2(\Omega)}^2,
\end{equation}
with $z_n^\delta=\tau^{-1}\int_{t_{n-1}}^{t_n} z^\delta(t) \mathrm{d}t$, $K_h = K\cap V_h$, 
where $u_h^n \equiv u^n_h(q_h)\in V_{h0}$ satisfies $u_h^0(q_h)=P_hu_0$ and
\begin{equation}\label{equ:finite element problem in parabolic system}	
	(\partial_\tau u_h^n,\varphi_h)+(\nabla u^n_h,\nabla\varphi_h)+(q_hu^n_h,\varphi_h)=(f^n,\varphi_h), \quad\forall \varphi_h\in V_{h0},\ n=1,\ldots,N.
\end{equation}
Problem \eqref{equ:discrete minimizing problem in parabolic system}---\eqref{equ:finite element problem in parabolic system} is a finite-dimensional constrained optimization problem, and the existence of a minimizer $q_h^*$ follows easily in view of the norm equivalence in finite-dimensional spaces. Below we give an error analysis of the approximation $q_h^*$.

In the analysis, we use extensively the Ritz projection operator $R_h: H^1_0(\Omega)\to V_{h0}$ defined by
\begin{equation*}
	(\nabla u(q),\nabla \varphi_h)+(qu(q),\varphi_h)=(\nabla R_hu(q),\nabla\varphi_h)+(qR_hu(q),\varphi_h),\quad \forall\varphi_h\in V_{h0}.
\end{equation*}
The following approximation result holds \cite{Thome2006GalerkinFE}:
\begin{equation}\label{inequ:R_h on H^2}
	\|v-R_hv\|_{L^2(\Omega)}+h\|v-R_hv\|_{H^1(\Omega)}\leq  Ch^2\|v\|_{H^2(\Omega)}, \quad\forall v\in H^2(\Omega)\cap H^1_0(\Omega).
\end{equation}

\subsection{Error estimates}
First, we derive a weighted $L^2(\Omega)$-error estimates of the approximation $q_h^*$. Throughout we make following assumption on the problem data. Under Assumption \ref{assum:Assumption in parabolic system}, by the standard parabolic regularity theory \cite{Evans2010PartialDE}, the weak solution $u(q^\dag)\in L^2(0,T;H^1_0(\Omega)) \cap H^1(0,T;H^{-1}(\Omega))$ to  problem \eqref{equ:variational problem in parabolic system} belongs to $H^1(0,T;H_0^1(\Omega))\cap L^\infty(0,T;H^2(\Omega)) $, and then by Sobolev embedding theorem \cite{Adams2003Sobolev}, $u(q^\dag)\in L^\infty(0,T;L^\infty(\Omega))$.

\begin{assumption}\label{assum:Assumption in parabolic system}
$q^{\dagger}\in H^2(\Omega)\cap K$, \ $u_0\in H^2(\Omega)\cap H_0^1(\Omega)$ and $f\in H^1(0,T;L^2(\Omega))$.
\end{assumption}

Now we can state the main result of this section, i.e., a weighted $L^2(\Omega)$ error estimate.
\begin{theorem}\label{thm:error estimates of q-q_h with weight in parabolic system}
Let Assumption \ref{assum:Assumption in parabolic system} be fulfilled. Let $q^{\dagger}\in K$ be the exact coefficient, $u(q^{\dagger})$ be the solution of problem \eqref{equ:variational problem in parabolic system}, and $q_h^*\in K_h$ be a minimizer of problem \eqref{equ:discrete minimizing problem in parabolic system}-\eqref{equ:finite element problem in parabolic system}. Then with $\eta=\tau+h^2+\delta+\alpha^{\frac{1}{2}}$, there exists a constant $C$ depending on $q^\dagger$ such that
		\begin{equation}\label{equ:error estimates of q-q_h in parabolic system}
		\tau^3\sum_{j=N_{0}+1}^{N}\sum_{i=N_{0}+1}^{j}\sum_{n=i}^{j}\|(q^{\dagger}-q_h^*)u^n(q^\dagger)\|_{L^2(\Omega)}\leq  C({\tau^{\frac{1}{2}}}+{h^{\frac{1}{2}}}+{\alpha^{\frac{1}{4}}}+{\min(h^{\frac{1}{2}}+h^{-\frac{1}{2}}\eta^{\frac{1}{2}},1)})\alpha^{-\frac{1}{4}}\eta^{\frac{1}{2}}.
		\end{equation}
\end{theorem}

The overall proof strategy is similar to the elliptic case in Section \ref{sec:err-ell}, but it is more involved due to the presence of time derivative. It relies heavily on the following preliminary results. The lengthy proofs are similar to the elliptic case, and hence are deferred to the appendix.
\begin{lemma}\label{lemma:error estimates of u(q)-u_h(q) in L^2 in parabolic system}
	Let Assumption \ref{assum:Assumption in parabolic system} be fulfilled. Then for sufficiently small $\tau$, there holds
	\begin{equation}\label{inequ:error estimates of u(q)-u_h(q) in L^2 in parabolic system}
		\tau\sum_{n=1}^{N}\|u^n(q^{\dagger})-u^n_h(q^{\dagger})\|^2_{L^2(\Omega)}\leq  C(\tau^2+h^4).
	\end{equation}
\end{lemma}
\begin{lemma}\label{lemma:error estimates of u(q)-u_h(Pi_hq) in L^2 in parabolic system}
	Let Assumption \ref{assum:Assumption in parabolic system} be fulfilled. Then for small $\tau$, there holds
	\begin{equation}\label{inequ:error estimates of u(q)-u_h(Pi_hq) in L^2 in parabolic system}
		\tau\sum_{n=1}^{N}\|u^n(q^{\dagger})-u^n_h(\Pi_hq^{\dagger})\|^2_{L^2(\Omega)}\leq  C(\tau^2+h^4).
	\end{equation}
\end{lemma}
\begin{lemma}\label{lemma:error estimate of u(q)-u_h(q_h^*) and grad q_h^* in L^2 in parabolic system}
Let the assumption in Theorem \ref{thm:error estimates of q-q_h with weight in parabolic system} be fulfilled. Then there holds
	\begin{equation*}\label{inequ:error estimate of u(q)-u_h(q_h^*) and grad q_h^* in L^2 in parabolic system}
		\tau\sum_{n=N_{0}}^N\|u^n(q^\dagger)-u^n_h(q_h^*)\|^2_{L^2(\Omega)}+\alpha\|\nabla q_h^*\|^2_{L^2(\Omega)}\leq  C(\tau^2+h^4+\delta^2+\alpha).
	\end{equation*}
\end{lemma}

Now we can state the proof of Theorem \ref{thm:error estimates of q-q_h with weight in parabolic system}.
\begin{proof}
Let $u\equiv u(q^\dag)$. For any $\varphi\in H_0^1(\Omega)$, we have the following splitting
\begin{align*}
	&((q^\dagger-q_h^*)u^n,\varphi)
	=((q^\dagger-q_h^*)u^n,\varphi-P_h\varphi)+((q^\dagger-q_h^*)u^n,P_h\varphi)  \\
	=&((q^\dagger-q_h^*)u^n,\varphi-P_h\varphi)+(\partial_\tau u_h^n(q_h^*)-\partial_tu^n,P_h\varphi)\\
	&+(\nabla(u^n_h(q_h^*)-u^n),\nabla P_h\varphi)+(q_h^*(u^n_h(q_h^*)-u^n),P_h\varphi)=: \sum_{i=1}^4 {\rm I}_i^n.
\end{align*}
Let $\varphi=\varphi^n=(q^\dagger-q_h^*)u^n$. Then
$\nabla\varphi^n = (\nabla q^\dagger -\nabla q_h^*)u^n+(q^\dagger-q^*_h)\nabla u^n.$ By Assumption \ref{assum:Assumption in parabolic system} and the a priori regularity $ u\in L^\infty(0,T;L^\infty(\Omega))$, there  holds 
\begin{equation}\label{inequ:estimates for varphin}
    \|\varphi^n\|\leq C\quad\mbox{and}\quad\|\nabla\varphi^n\|_{L^2(\Omega)}\leq C(1+\|\nabla q_h^*\|_{L^2(\Omega)}),\quad n=1,\ldots,N,
\end{equation}
which implies $\varphi^n\in H^1_0(\Omega)$.
Using Assumption \ref{assum:Assumption in parabolic system}, Lemma \ref{lemma:error estimate of u(q)-u_h(q_h^*) and grad q_h^* in L^2 in parabolic system}, and repeating the argument of Theorem \ref{thm:error estimates of q-q_h with weight}, we deduce
\begin{equation*}
	\sum_{n=N_0}^N|{\rm I}^n_1|\leq  C{h}\alpha^{-\frac{1}{2}}\eta,\  \tau\sum_{n=N_{0}}^N|{\rm I}_3^n|\leq  C{\min(h+h^{-1}\eta,1))}\alpha^{-\frac{1}{2}}\eta,\ \mbox{and}\ \tau\sum_{n=N_{0}}^N|{\rm I}_4^n|\leq  C{\eta},
\end{equation*}
where the constant $C$ depends on $q^\dagger$. To estimate ${\rm I}_2^n$, we split it into
\begin{equation*}
  	{\rm I}_2^n=(\partial_\tau (u_h^n(q_h^*)-u^n),P_h\varphi^n)+(\partial_\tau u^n-\partial_tu^n,P_h\varphi^n)=: {\rm I}_{2,1}^n+{\rm I}_{2,2}^n.
\end{equation*}
Let $\varphi=P_h\varphi^n$ in \eqref{equ:integration of variational formulation in papabolic system} and \eqref{equ:variational problem in parabolic system}.
By Assumption \ref{assum:Assumption in parabolic system}, $f\in H^1(0,T;L^2(\Omega)),$ we have the a priori regularity $u\in H^1(0,T;H^1(\Omega))$. Consequently, the proof of Lemma \ref{lemma:error estimates of u(q)-u_h(q) in L^2 in parabolic system} yields
\begin{align*}
    \tau\sum_{n=N_{0}}^N\|\bar{f}^n-f^n\|^2_{L^2(\Omega)}\leq C\tau^2,\quad 
    \tau\sum_{n=N_{0}}^N\|\bar{u}^n-u^n\|^2_{H^1(\Omega)}\leq C\tau^2, \quad
    \tau\sum_{n=N_{0}}^N\|\bar{u}^n-u^n\|^2_{L^2(\Omega)}\leq C\tau^2.
\end{align*}
This, the $L^2(\Omega)$-stability of $P_h$ in \eqref{inequ:P_h}, the Cauchy-Schwarz inequality and \eqref{inequ:estimates for varphin} imply
\begin{align*}
\Big|\tau\sum_{n=N_{0}}^N{\rm I}_{2,2}^n\Big|  &=\Big|\tau\sum_{n=N_{0}}^N(\bar{f}^n-f^n,P_h\varphi^n)-\tau\sum_{n=N_{0}}^N(\nabla(\bar{u}^n-u^n),\nabla P_h\varphi^n)-\tau\sum_{n=N_{0}}^N(q^\dagger(\bar{u}^n-u^n),P_h\varphi^n)\Big| \\
	&\leq C\Big(\tau\sum_{n=N_{0}}^N\|\bar{f}^n-f^n\|^2_{L^2(\Omega)}\Big)^{\frac{1}{2}}+\Big(\tau\sum_{n=N_{0}}^N\|\bar{u}^n-u^n\|^2_{H^1(\Omega)}\Big)^{\frac{1}{2}}\max_{n=N_0,\ldots,N}\|\nabla\varphi^n\|_{L^2(\Omega)}\\ &\quad +\Big(\tau\sum_{n=N_{0}}^N\|\bar{u}^n-u^n\|^2_{L^2(\Omega)}\Big)^\frac{1}{2}
	\leq C{\tau}\alpha^{-\frac{1}{2}}\eta. 
\end{align*}
Next we bound the term ${\rm I}_{2,1}^n$. For $N_{0}+1\leq  i\leq  j\leq  N$, the summation by parts formula leads to
\begin{align*}
	\tau\sum_{n=i}^j{\rm I}_{2,1}^n&=
	(u_h^j(q^*_h)-u^j,P_h\varphi^j)-(u_h^{i-1}(q^*_h)-u^{i-1},P_h\varphi^{i-1}) 
	-\tau\sum_{n=i}^j(u_h^{n-1}(q^*_h)-u^{n-1},\partial_\tau P_h\varphi^n).
\end{align*}
Using Assumption \ref{assum:Assumption in parabolic system}, Lemma \ref{lemma:error estimate of u(q)-u_h(q_h^*) and grad q_h^* in L^2 in parabolic system}, the $L^2(\Omega)$ stability of $P_h$ and the Cauchy-Schwarz inequality, we have
\begin{align*}
 &|\tau^2\sum_{j=N_{0}+1}^N\sum_{i=N_{0}+1}^j(u_h^j(q^*_h)-u^j,P_h\varphi^j)-(u_h^{i-1}(q^*_h)-u^{i-1},P_h\varphi^{i-1})| \\
 &\leq  C\Big(\tau\sum_{n=N_{0}}^N\|u_h^n(q^*_h)-u^n\|_{L^2(\Omega)}^2\Big)^{\frac{1}{2}}\leq  C{\eta}.
\end{align*}
Next, by the $L^2(\Omega)$ stability of $P_h$ in \eqref{inequ:P_h}, the box constraint and the Cauchy-Schwarz inequality, we have
\begin{equation*}
    \|\partial_\tau P_h\varphi^n\|_{L^2(\Omega)} = \Big\|P_h(\tau^{-1}\int_{t^{n-1}}^{t^n}(q^\dagger-q_h^*)\partial_tu\mathrm{d}t)\Big\|_{L^2(\Omega)}\leq  C\Big\|\tau^{-\frac{1}{2}}\Big(\int_{t^{n-1}}^{t^n}|\partial_tu|^2\mathrm{d}t\Big)^{\frac{1}{2}}\Big\|_{L^2(\Omega)}.
\end{equation*}
Since $u\in H^1(0,T;H^1_0(\Omega))$, we deduce
\begin{equation*}
    \tau\sum_{n=N_{0}+1}^N\|\partial_\tau P_h\varphi^n\|_{L^2(\Omega)}^2\leq \sum_{n=N_{0}+1}^N\|\partial_tu\|^2_{L^2(t^{n-1},t^n;L^2(\Omega))}=\|\partial_tu\|^2_{L^2(T_0,T;L^2(\Omega))}\leq C.
\end{equation*}
This and the Cauchy-Schwarz inequality imply
\begin{align*}
	\Big|\tau\sum_{n=i}^j(u_h^{n-1}(q^*_h)-&u^{n-1},\partial_\tau P_h\varphi^n)\Big| 
	\leq \tau\sum_{n=i}^j\|u_h^{n-1}(q^*_h)-u^{n-1}\|_{L^2(\Omega)}\|\partial_\tau P_h\varphi^n\|_{L^2(\Omega)} \\
	&\leq \Big(\tau\sum_{n=N_{0}+1}^N\|u_h^{n-1}(q^*_h)-u^{n-1}\|_{L^2(\Omega)}^2\Big)^{\frac{1}{2}}\Big(\tau\sum_{n=N_{0}+1}^N\|\partial_\tau P_h\varphi^n\|_{L^2(\Omega)}^2\Big)^\frac{1}{2}
	\leq  C{\eta},
\end{align*}
and hence there holds
\begin{equation*}
	|\tau^3\sum_{j=N_{0}+1}^N\sum_{i=N_{0}+1}^j\sum_{n=i}^j(u_h^{n-1}(q^*_h)-u^{n-1},\partial_\tau P_h\varphi^n)|\leq  C{\eta}.
\end{equation*}
Combining the preceding estimates gives
\begin{equation*}
	\tau^3\sum_{j=N_{0}+1}^N\sum_{i=N_{0}+1}^j\sum_{n=i}^j\|(q^\dagger-q_h^*)u^n\|^2_{L^2(\Omega)}\leq  C(\tau+h+\alpha^{\frac{1}{2}}+\min(h+h^{-1}\eta,1))\alpha^{-\frac{1}{2}}\eta.
\end{equation*}
This completes the proof of the theorem.
\end{proof}

Now we establish interior and global $L^2$ error estimations.
\begin{corollary}\label{coro:interior error estimates in $L^2$-norm of parabolic problem}
Let Assumption \ref{assum:Assumption in parabolic system} be fulfilled, the source $f\not\equiv  0$ be nonnegative  a.e. in $\Omega\times(0,T)$ and $u_0\geq  0$ a.e. in $\Omega$. Then for any compact subset $\Omega'\subset\subset\Omega$ with {\rm dist}$(\overline{\Omega'},\partial\Omega)>0$, there exists a positive constant $C$, depending on {\rm dist}$(\overline{\Omega'},\partial\Omega)$ and $q^\dagger$, such that
\begin{equation*}
\|q^\dagger-q^*_h\|_{L^2(\Omega')}\leq  C(\tau^\frac{1}{2}+h^{\frac{1}{2}}+\alpha^{\frac{1}{4}}+\min(h^\frac12+h^{-\frac{1}{2}}\eta^{\frac{1}{2}},1))\alpha^{-\frac{1}{4}}\eta^{\frac{1}{2}}.
\end{equation*}

\end{corollary}
\begin{proof}
By the standard parabolic maximum principle (see, e.g., \cite{Nirenberg:1953,Friedman:1958}), $f\geq0$ a.e. in $\Omega\times(0,T)$ and $u_0\geq0$ in $\Omega$ imply $u(q^\dagger)\geq  0\ \mbox{a.e. in}\ \Omega\times(0,T)$. 
Let $w\in L^2(0,T;H_0^1(\Omega))\cap H^1(0,T;H^{-1}(\Omega))$ with $w(0)= \frac{u_0}{2}$ solve
\begin{equation*}	
	(\partial_tw,\varphi)+(\nabla w,\nabla\varphi)+(c_1w,\varphi)=(f,\varphi), \quad\forall \varphi\in H^1_0(\Omega),\ \mbox{a.e.}\ t\in(0,T).
\end{equation*}
Then $ v= u(q^\dagger)-w$ satisfies $v(0)=\frac{u_0}{2}$ and
\begin{equation*}	
	(\partial_t v,\varphi)+(\nabla v,\nabla\varphi)+(q^\dag v,\varphi)=((c_1-q^\dag)u(q^\dag),\varphi), \quad\forall \varphi\in H^1_0(\Omega),\ \mbox{a.e.}\ t\in(0,T).
\end{equation*}
Then the weak maximum principle for parabolic problem implies  $v\geq 0$, i.e., $u(q^\dagger)\geq  w$  a.e. in $\Omega\times(0,T)$. By Theorem \ref{thm:error estimates of q-q_h with weight in parabolic system}, there holds
    \begin{equation*}
			\tau^3\sum_{j=N_{0}+1}^{N}\sum_{i=N_{0}+1}^{j}\sum_{n=i}^{j}\|(q^{\dagger}-q_h^*)w^n\|_{L^2(\Omega)}\leq  C(\tau^{\frac{1}{2}}+h^{\frac{1}{2}}+\alpha^{\frac{1}{4}}+{\min(h^\frac12+h^{-\frac{1}{2}}\eta^{\frac{1}{2}},1)})\alpha^{-\frac{1}{4}}\eta^{\frac{1}{2}}.
		\end{equation*}
Then by \cite[Theorem 2.4]{Bobkov2018OnMA}, we have $w\in C((0,T];C(\overline{\Omega}))$ and $w>0\ \mbox{in}\ \Omega\times(0,T)$. The assertion follows from the continuity of $w$ and the compactness of $\Omega'$.
\end{proof}

We also have the following global $L^2(\Omega)$ estimate by the positive condition \eqref{eqn:positivity-parab}.
\begin{corollary}\label{coro:standard L2 in para}
Let conditions in Theorem \ref{thm:error estimates of q-q_h with weight in parabolic system} be fulfilled and condition \eqref{eqn:positivity-parab} hold. Then there holds
\begin{equation*}
    \|q^\dagger-q_h^*\|_{L^2(\Omega)}\leq C ((\tau^\frac{1}{2}+h^{\frac{1}{2}}+\alpha^{\frac{1}{4}}+{\min(h^\frac12+h^{-\frac{1}{2}}\eta^{\frac{1}{2}},1)})\alpha^{-\frac{1}{4}}\eta^{\frac{1}{2}})^\frac{1}{(1+2\beta)}.
\end{equation*}
\end{corollary}

\begin{remark}\label{remark:index match in parabolic system}
For the choice $\alpha\sim\delta^2$, $\tau\sim\delta$ and $h\sim\delta^{\frac{1}{2}}$, we obtain the following interior estimate: $\|q^{\dagger}-q_h^*\|_{L^2(\Omega')}\leq  C\delta^{\frac{1}{4}}.$
Likewise, under the positivity condition, there holds 
$\|q^{\dagger}-q_h^*\|_{L^2(\Omega)}\leq  C\delta^{\frac{1}{4(1+2\beta)}}.$
\end{remark}

\section{Numerical experiments and discussions}
\label{sec:numer}
Now we present numerical experiments for elliptic and parabolic inverse problems to illustrate the analysis. We solve the discrete optimization system by the conjugate gradient method \cite{Alifanov:1995}, with the gradient computed by the standard adjoint technique. Despite the regularized functional being nonconvex, the method converges relatively robustly with the given initial guess, and the convergence is achieved within tens of iterations. The lower and upper bounds of the box constraint $K$ are taken to be $0.4$ and $2$, respectively. In the elliptic case, the noise data $z^\delta$ is generated by
\begin{equation*}
	z^\delta(x)=u(q^\dagger)(x)+\epsilon\|u(q^\dagger)\|_{L^\infty(\Omega)}\xi(x),\quad x\in \Omega,
\end{equation*}
where $\xi$ follows the standard normal Gaussian distribution and $\epsilon>0$ is the relative noise level. The noise data $z^\delta(x,t)$ in the parabolic case is generated similarly. In the computation, we set the parameters $\alpha\sim\delta^2$, $h\sim\sqrt{\delta}$ and $\tau\sim\delta$ according to Remarks \ref{remark:index match in elliptic system} and \ref{remark:index match in parabolic system}.

The first test is about the elliptic case.
\begin{example}\label{exam:elliptic}
We consider the following two cases.
\begin{itemize}
\item[(a)] $\Omega=(0,1)$, $q=1+x(1-x)\sin(2\pi x)$ and $f\equiv1$. 
\item[(b)] $\Omega=(0,1)^2$, $q=1+y(1-y)\sin(\pi x)$ and $f\equiv1$. 
\end{itemize}
\end{example}

To study the convergence behavior of the discrete approximation $q_h^*$, we employ two different metrics, i.e.,  $e_q=\|q^\dagger-q^*_h\|_{L^2(\Omega)}$ and $e_u=\|u(q^\dagger)-u_h(q^*_h)\|_{L^2(\Omega)}$. Note that the error analysis provides a convergence $O(\delta^\frac14)$ at best (in the interior) for $e_q$, and the state approximation $e_u$ is predicted to be $O(\delta)$. The numerical results for Example \ref{exam:elliptic} are presented in Table \ref{tab:err-ell}. The convergence of $e_q$ and $e_u$ can be observed clearly as the noise level $\delta\to 0$, more precisely, in the one-dimensional case, with the behavior $e_q\sim\delta^{0.54}$ and $e_u\sim\delta^{1.01}$. These observations remain valid for the two-dimensional test in (b). The empirical rate of $e_q$ is much faster than the theoretical one, indicating potential suboptimality of the theoretical results in Corollaries \ref{coro:interior error estimates in $L^2$-norm of elliptic problem} and \ref{coro:standard L2 in ellip}. Figs. \ref{fig:recon-ell1d} and \ref{fig:recon-ell2d} show the numerical reconstruction for the examples with different noise levels. These plots clearly show the convergence of the reconstruction $q_h^*$ as the noise level $\delta$ tends to zero. Note the accuracy near the boundary is a bit worse in all cases.

\begin{table}[htpb]
	\caption{Numerical results for Example \ref{exam:elliptic}.}
	\begin{center}\label{tab:err-ell}
		\subfigure[initialized with $\alpha=2.00${\rm e-7} and $h=2.00${\rm e-2}]{\begin{tabular}{c|cccccccc}
			\toprule
			$\epsilon$ & 2.00e-2 & 5.00e-3 & 1.25e-3 & 3.13e-4 & 7.81e-5 & 1.95e-5 & rate\\
			\midrule
			$e_q$ & 1.30e-1 & 1.20e-1 & 2.85e-2 & 2.57e-2 & 1.02e-2 & 5.98e-3 & 0.54 \\
			$e_u$ & 2.71e-4 & 6.81e-5 & 1.80e-5 & 3.43e-6 & 9.78e-7 & 1.72e-7 & 1.01 \\
			\midrule
			\end{tabular}}
		\subfigure[initialized with $\alpha=1.00${\rm e-6} and $h=1.00${\rm e-1}]{
      \begin{tabular}{c|cccccccc}
			\toprule
			$\epsilon$ & 1.00e-2 & 2.50e-3 & 6.25e-4 & 4.00e-4 & 1.00e-4 & rate\\
		\midrule
		$e_q$      & 1.16e-1 & 9.14e-2 & 7.07e-2 & 5.81e-2 & 3.55e-2 & 0.33\\
		$e_u$      & 8.36e-5 & 3.41e-5 & 1.30e-5 & 7.88e-6 & 1.60e-6 & 0.99\\
		\bottomrule
		\end{tabular}}
	\end{center}
\end{table}

\begin{figure}[hbpt!]
	\centering
	\subfigure[$\epsilon=2.00${\rm e-2}]{\includegraphics[width=0.3\textwidth,trim={4.5cm 9cm 4cm 9cm}, clip]{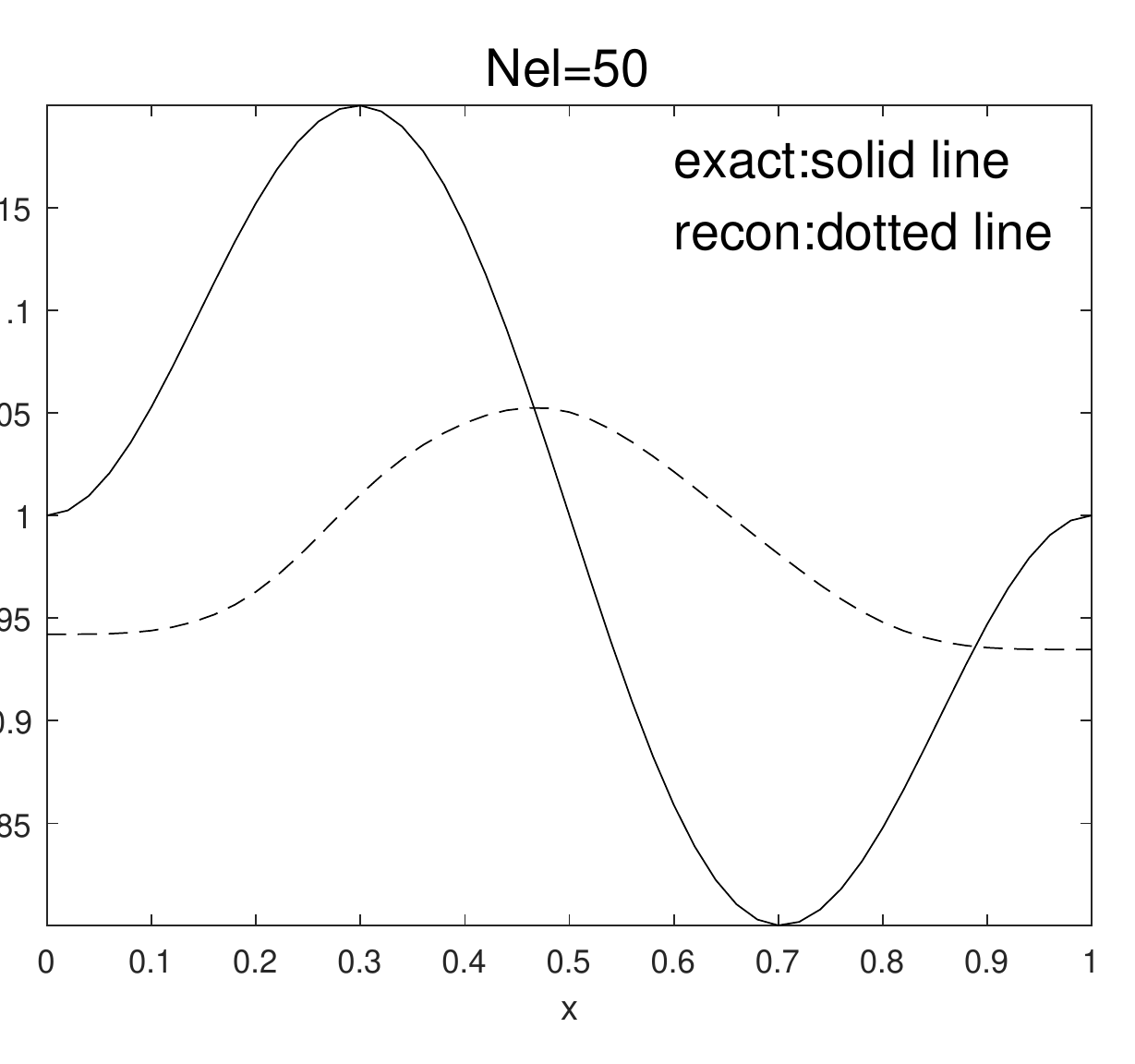}}
	\subfigure[$\epsilon=3.13${\rm e-4}]{
	\includegraphics[width=0.3\textwidth,trim={4.5cm 9cm 4cm 9cm}, clip]{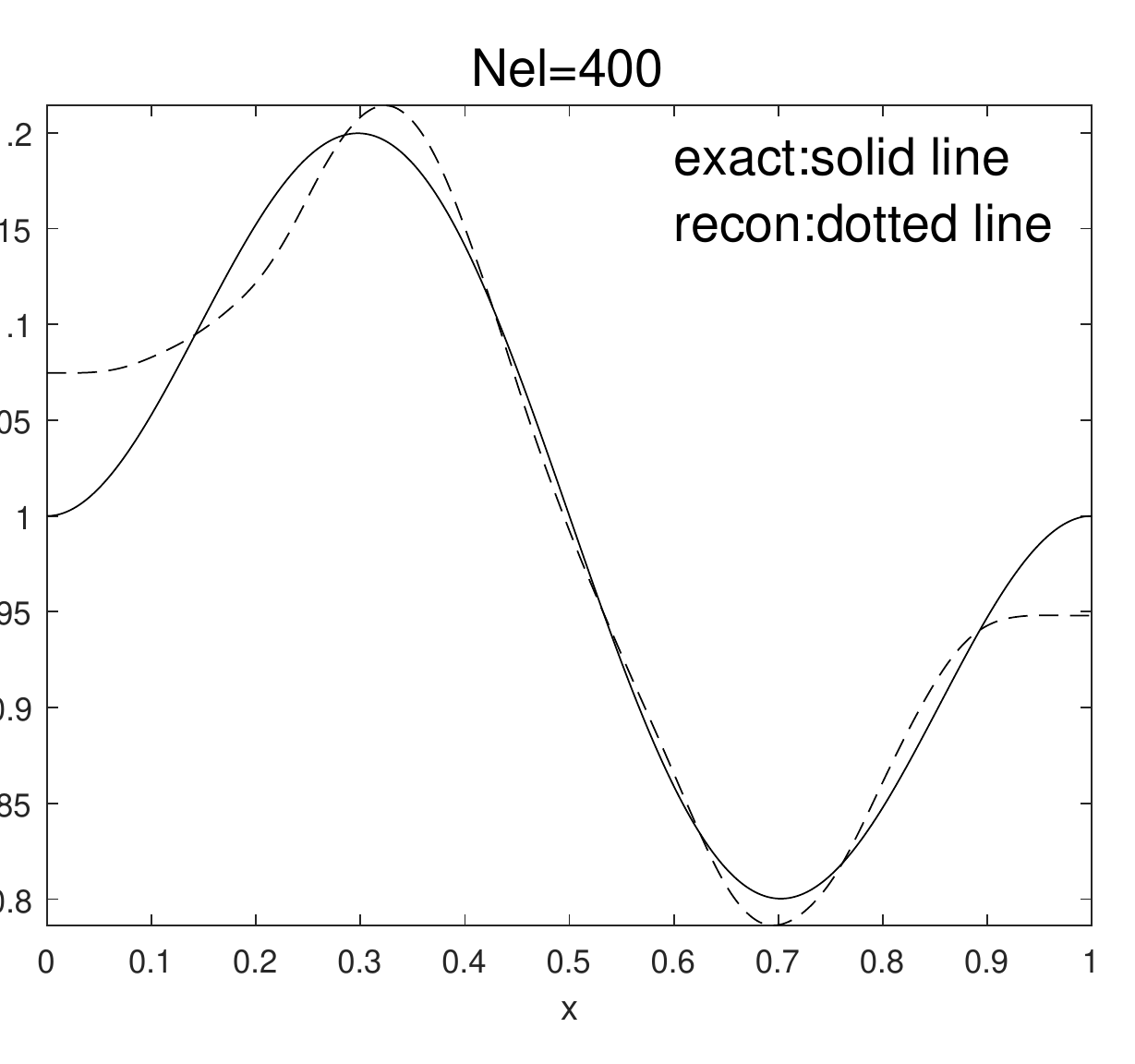}	}
	\subfigure[$\epsilon=1.95${\rm e-5}]{
	\includegraphics[width=.3\textwidth,trim={4.5cm 9cm 4cm 9cm}, clip]{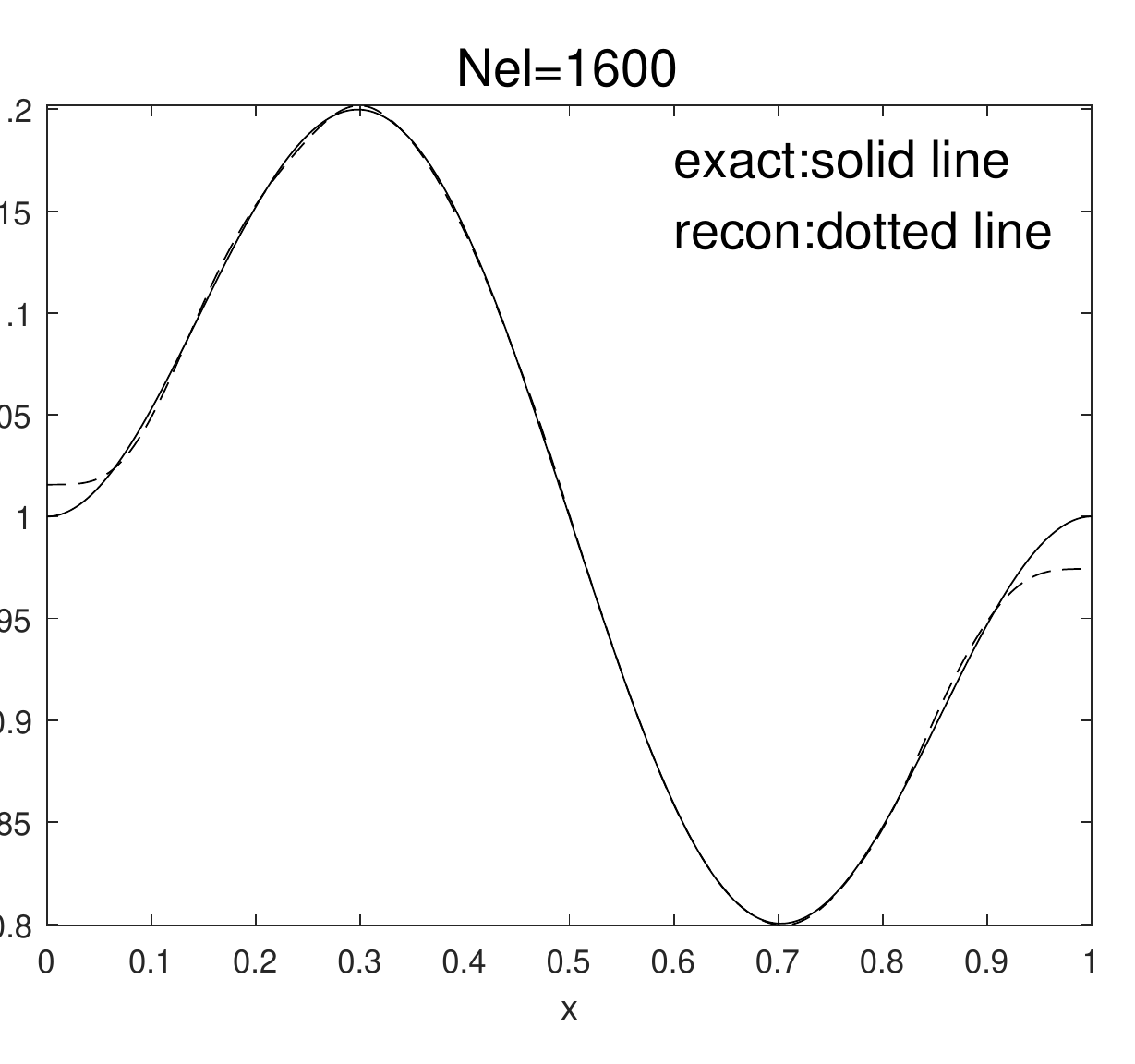}
	}
	\caption{Numerical reconstructions for Example \ref{exam:elliptic}(a).} \label{fig:recon-ell1d}
\end{figure}

\begin{figure}[htbp]
	\centering
	\subfigure[exact]{
	\includegraphics[width=0.29\textwidth,trim={4.5cm 9cm 4cm 9cm}, clip]{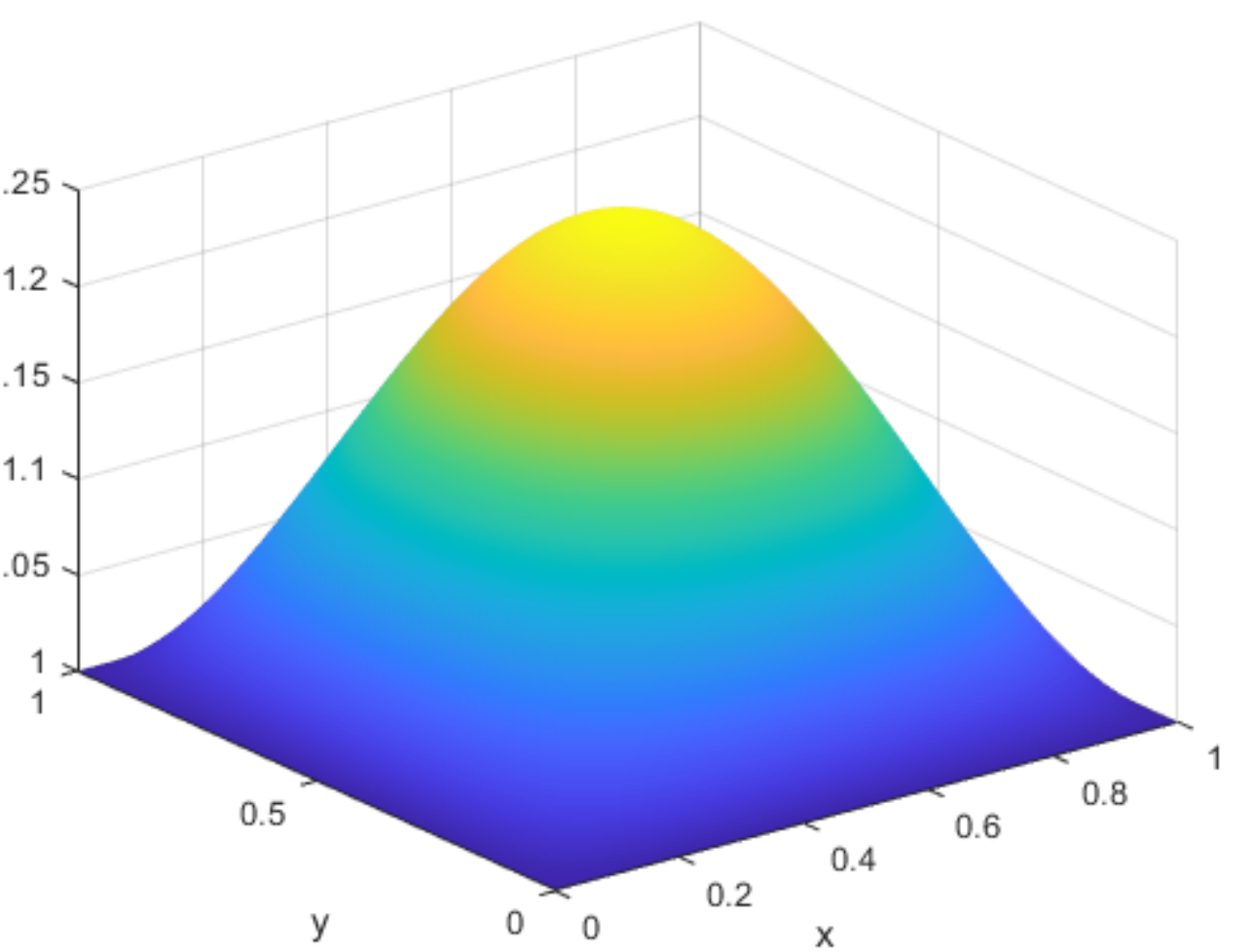}	}
	\subfigure[$\epsilon=1.00${\rm e-4}]{
    \includegraphics[width=0.29\textwidth,trim={4.5cm 9cm 4cm 9cm}, clip]{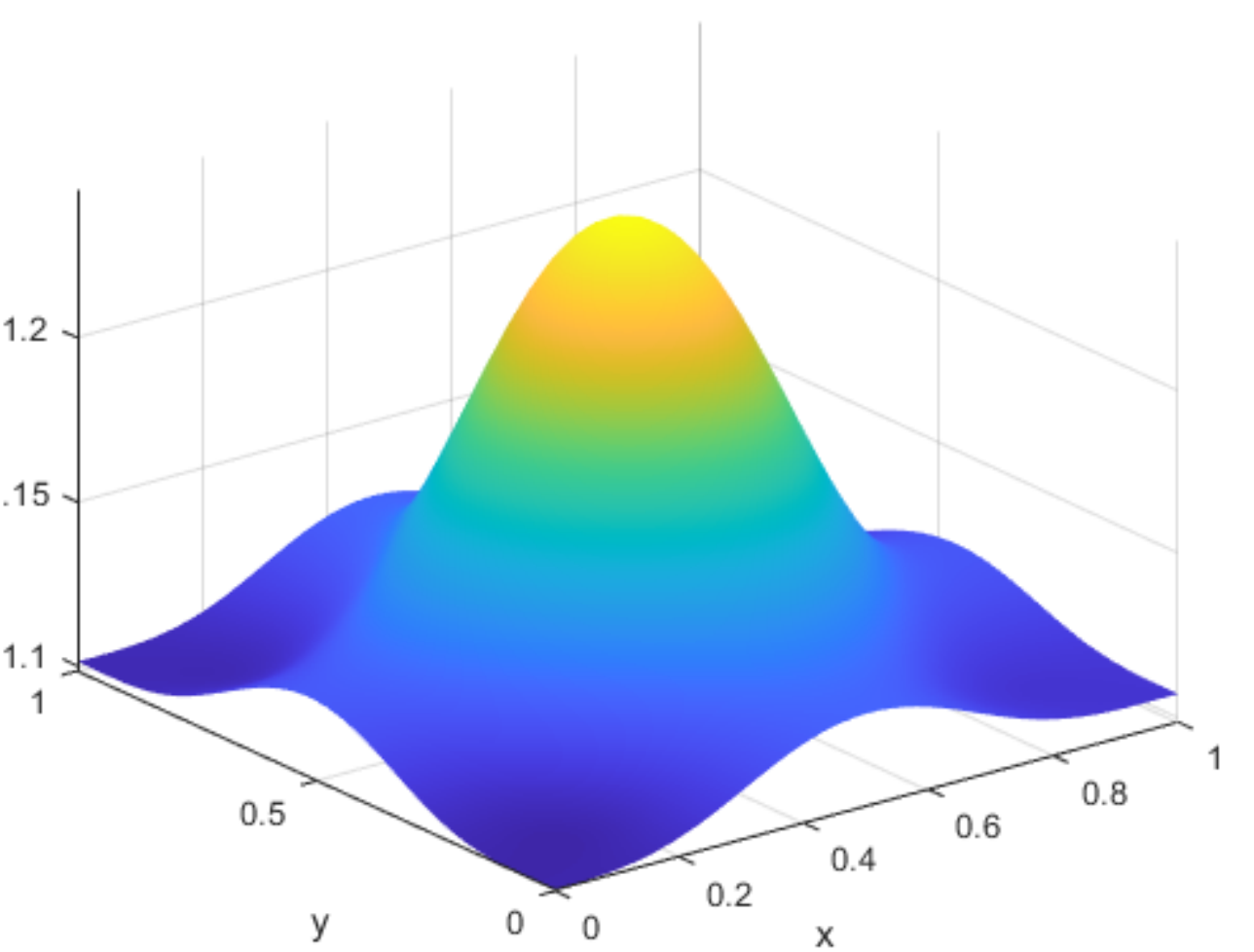}	}
	\subfigure[$\epsilon=3.91${\rm e-5}]{
    \includegraphics[width=0.29\textwidth,trim={4.5cm 9cm 4cm 9cm}, clip]{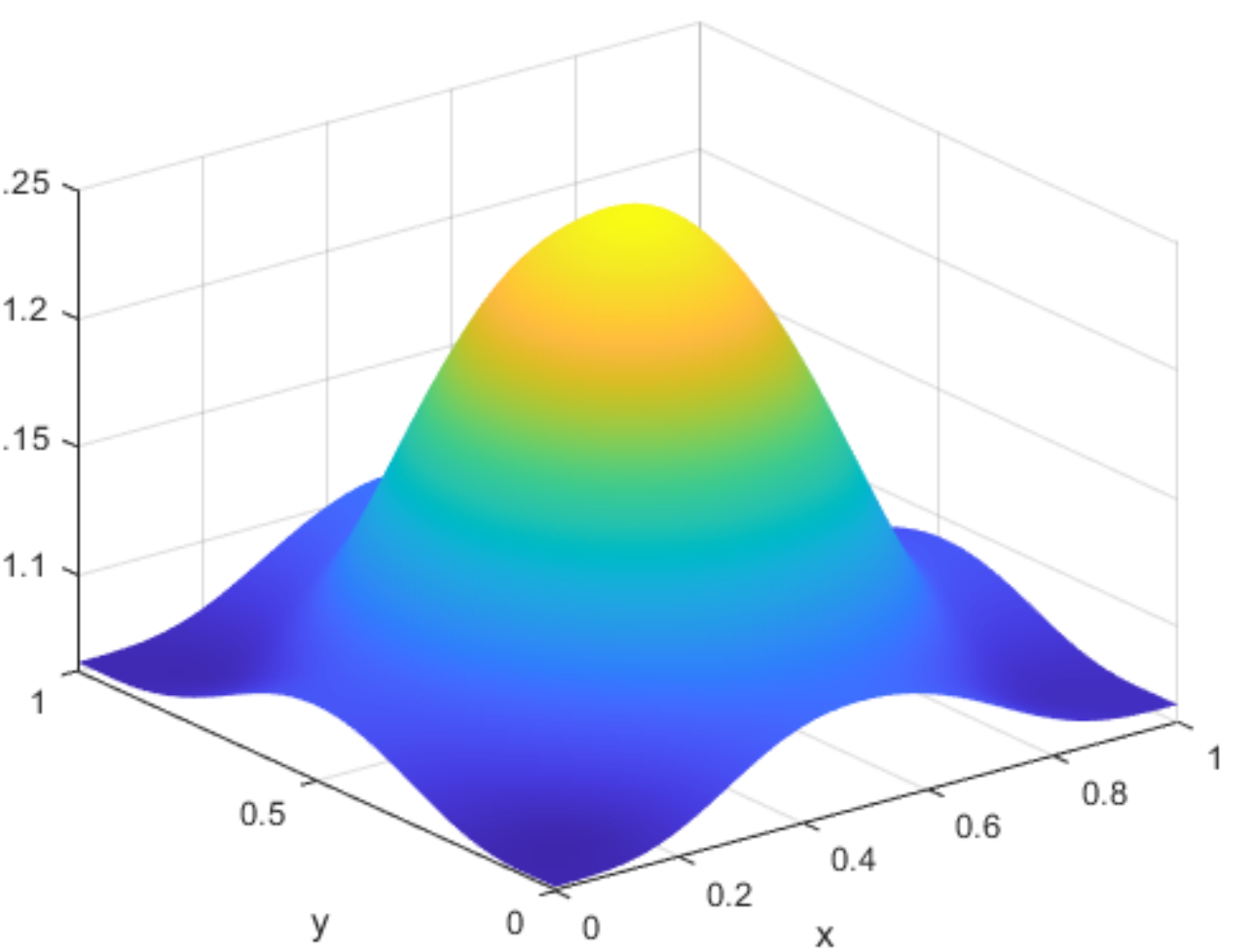}	}
	\caption{Numerical reconstructions for Example \ref{exam:elliptic}(b).}
	\label{fig:recon-ell2d}
\end{figure}

The second test is about the parabolic case.
\begin{example}\label{exam:para}
We consider the following two cases.
\begin{itemize}
    \item[(a)] $\Omega=(0,1)$, $T_0 = 0$, $T=0.01$, $q=1+\sin(2\pi x)/2$, $u_0 = \sin(\pi x)$ and $f\equiv 1$.
    \item[(b)] $\Omega=(0,1)^2$, $T_0=0$, $T=0.01$, $q=1+\sin(\pi x)\sin(\pi y)/2$, $u_0=\sin(\pi x)\sin(\pi y)$ and $f\equiv1$.  
\end{itemize}

\end{example}

Like in the elliptic case, we monitor the following two metrics, i.e., $e_q=\|q^\dagger-q^*_h\|_{L^2(\Omega)}$ and $e_u=(\tau\sum_{n=N_0}^{N}\|u^n(q^\dagger)-u^n_h(q^*_h)\|^2_{L^2(\Omega)})^\frac{1}{2}$. The numerical results for Example \ref{exam:para} are presented in Table \ref{tab:err-para1d} and Figs. \ref{fig:recon-para1d} and \ref{fig:recon-para2d}. Overall the convergence behaivour is very similar to the elliptic case: A steady convergence of both $e_q$ and $e_u$ is observed; The convergence rate of $e_u$ is slightly faster than first order; but the convergence rate of $e_q$ is again much higher than theoretical one. 
\begin{table}[htb!]
	\caption{Numerical results for Example \ref{exam:para}.}
	\begin{center}\label{tab:err-para1d}
		\subfigure[initialized with $\alpha=1.00${\rm e-8} and $h=2.00${\rm e-2}]{\begin{tabular}{c|cccccc}
			\toprule
			$\epsilon$ & 1.00e-2 & 2.50e-3 & 6.25e-4 & 1.56e-4 & 3.91e-5 & rate\\
			\midrule
			$e_q$      & 6.78e-1 & 1.73e-1 & 1.06e-1 & 6.69e-2 & 4.19e-2 & 0.50 \\
			$e_u$      & 1.61e-4 & 1.62e-5 & 3.68e-6 & 1.24e-6 & 2.31e-7 & 1.18 \\
			\bottomrule
		\end{tabular}}
		\subfigure[initialized with $\alpha=1.00${\rm e-6} and $h=1.00${\rm e-1}]{
		\begin{tabular}{c|cccccc}
			\toprule
			$\epsilon$ & 1.00e-2 & 2.50e-3 & 4.00e-4 & 1.00e-4 & rate \\
			\midrule
			$e_q$      & 9.35e-1 & 8.73e-1 & 4.11e-1 & 2.22e-1 & 0.30 \\
			$e_u$      & 4.31e-4 & 4.63e-5 & 2.15e-5 & 4.36e-6 & 1.06 \\
			\bottomrule
		\end{tabular}}
	\end{center}
\end{table}

\begin{figure}[htbp!]
	\centering
	\subfigure[$\epsilon=2.50${\rm e-3}]{
\includegraphics[width=0.3\textwidth,trim={4.5cm 9cm 4cm 9cm}, clip]{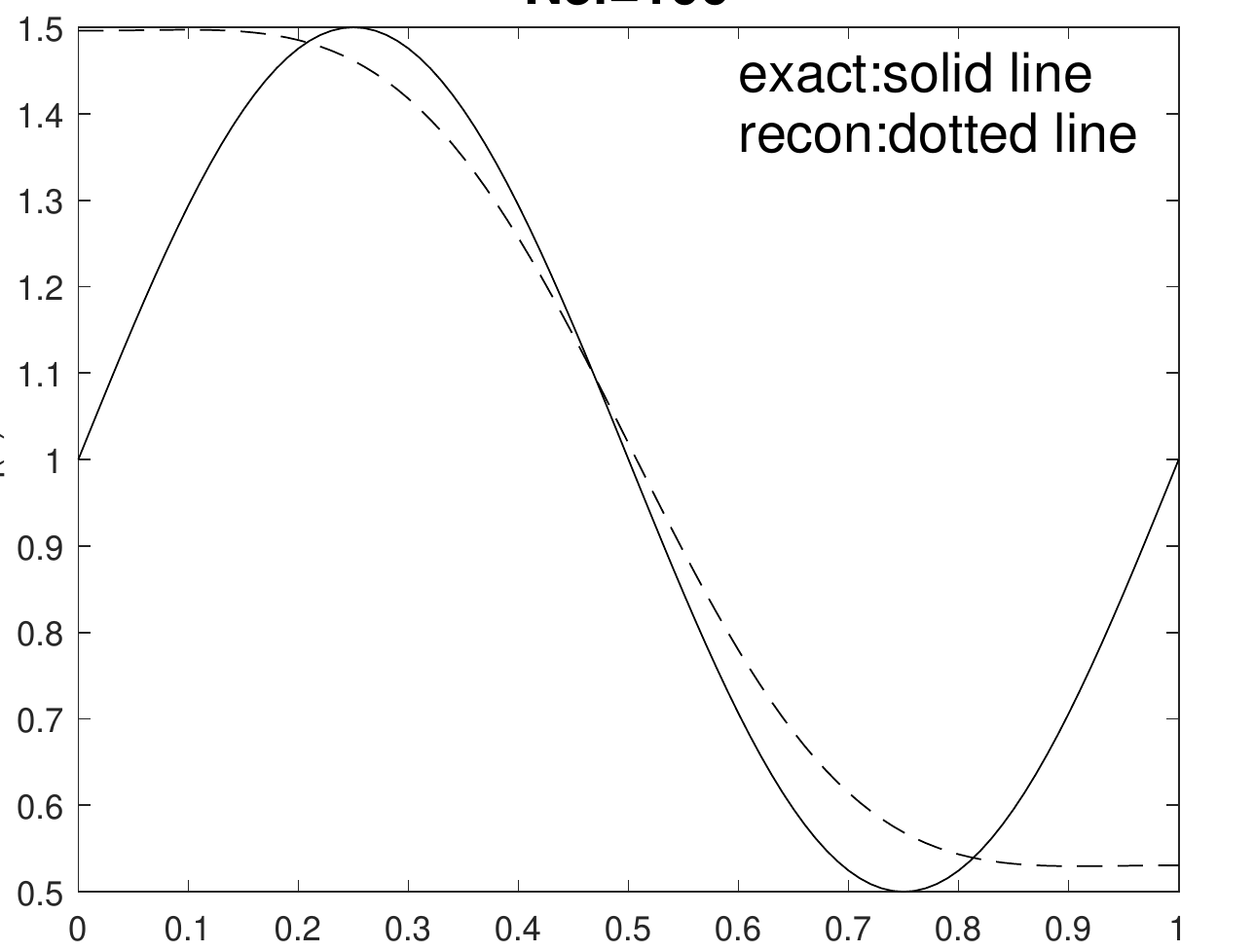}}
	\subfigure[$\epsilon=1.56${\rm e-4}]{
\includegraphics[width=0.3\textwidth,trim={4.5cm 9cm 4cm 9cm}, clip]{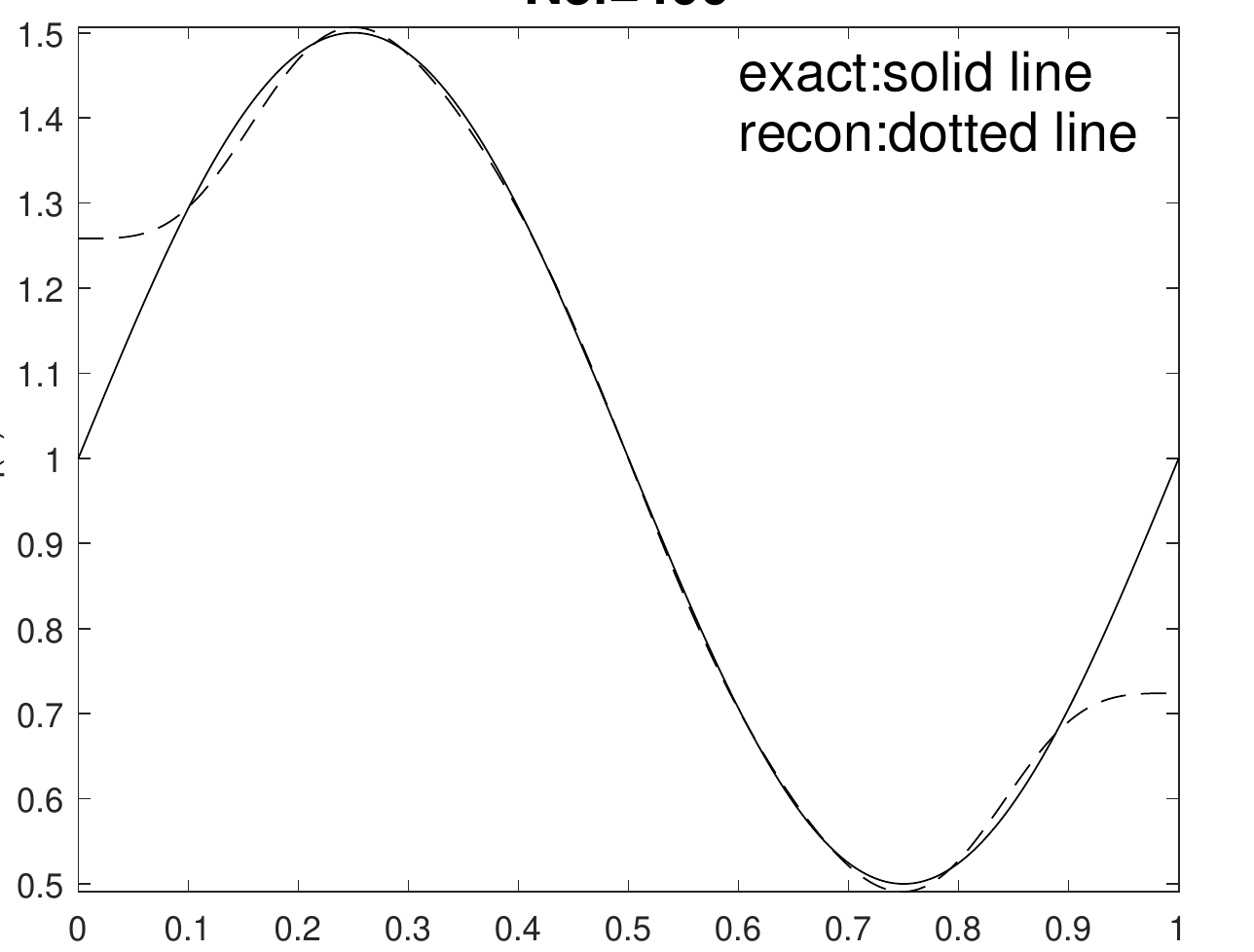}}
	\subfigure[$\epsilon= $3.91{\rm e-5}]{
\includegraphics[width=0.3\textwidth,trim={4.5cm 9cm 4cm 9cm}, clip]{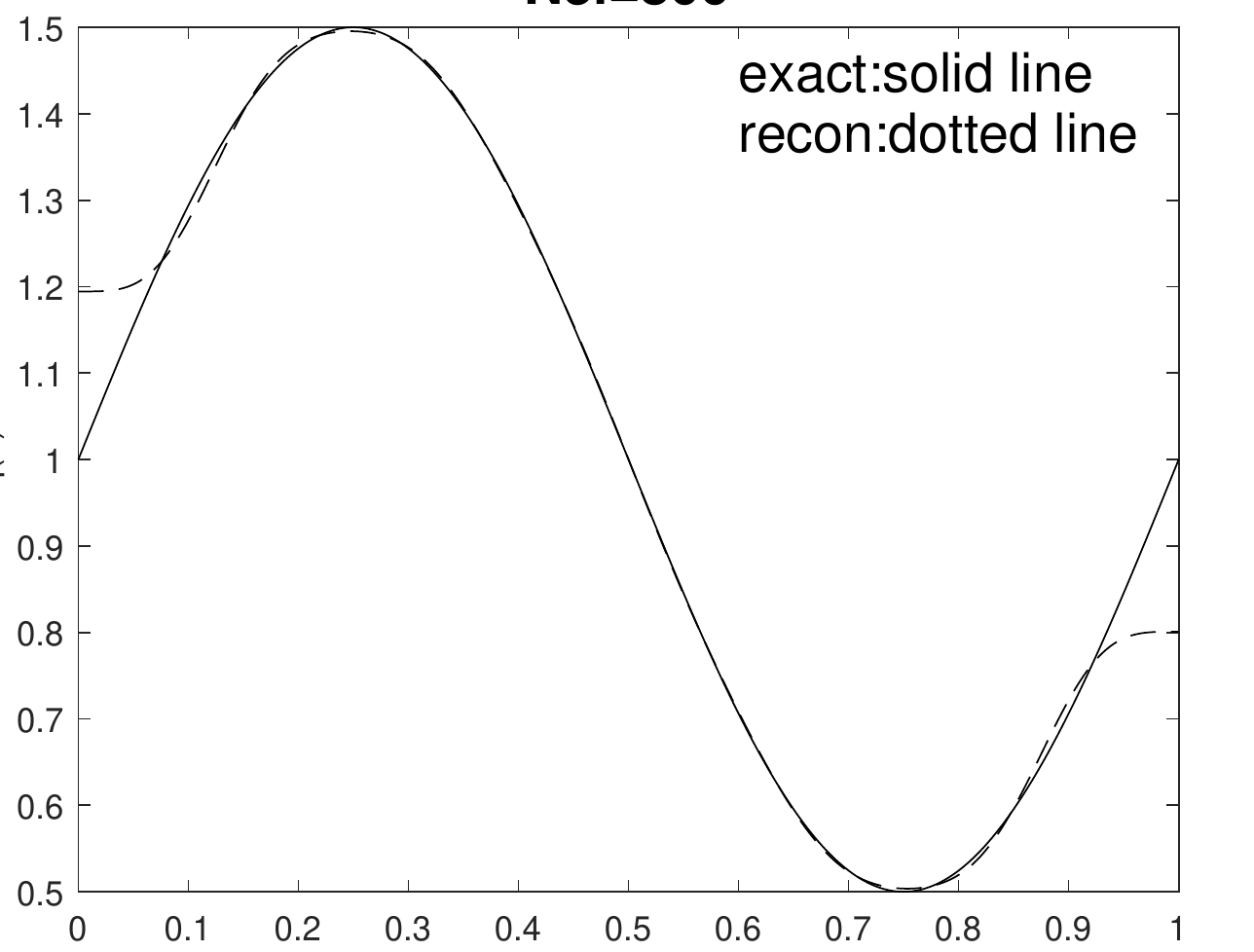}}
	\caption{Numerical reconstructions for Example \ref{exam:para}(a).}
	\label{fig:recon-para1d}
\end{figure}

\begin{figure}[htbp]
	\centering
\subfigure[exact]{\includegraphics[width=0.3\textwidth,trim={4.5cm 9cm 4cm 9cm}, clip]{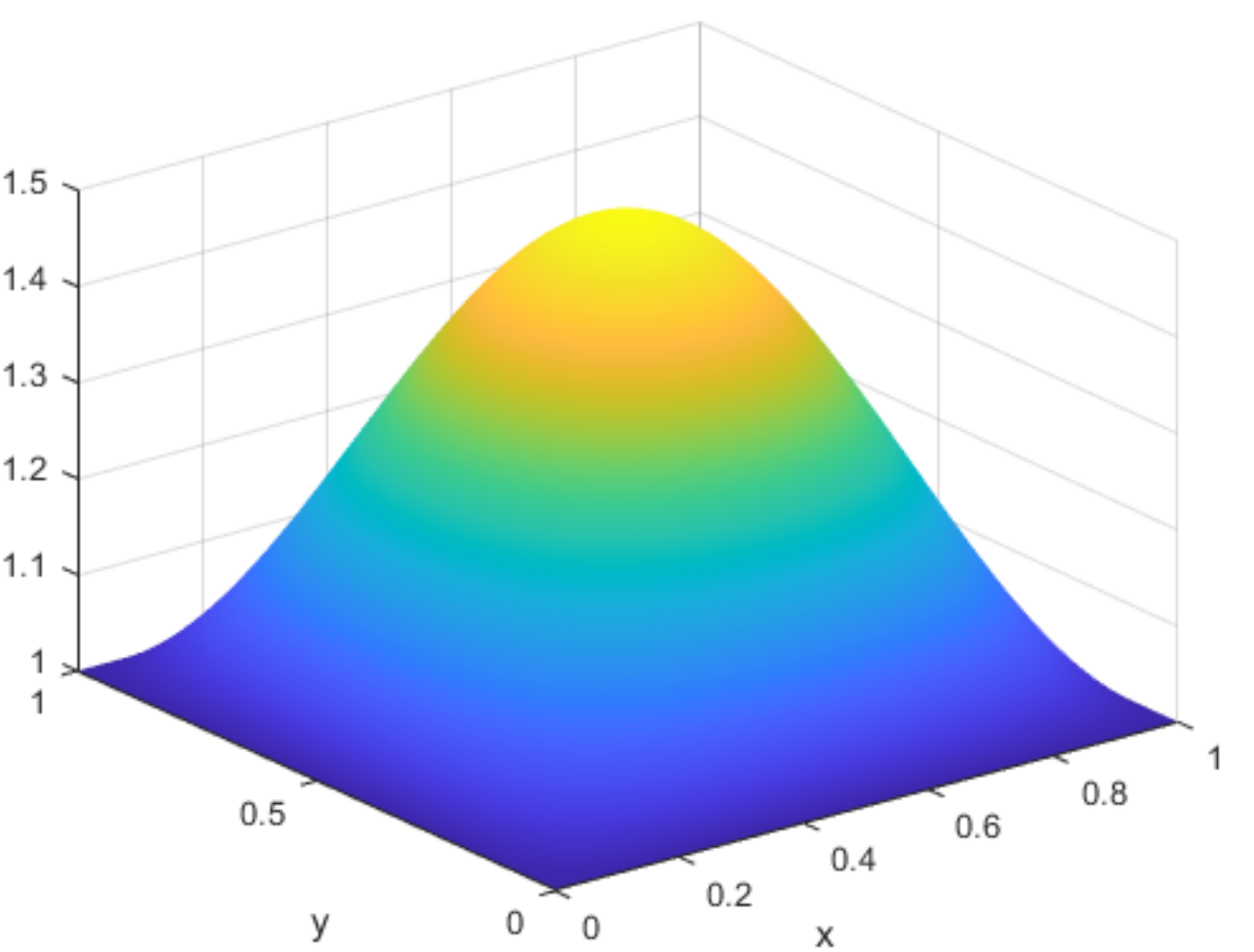}}
\subfigure[$\epsilon=4.00${\rm e-4}]{
\includegraphics[width=0.3\textwidth,trim={4.5cm 9cm 4cm 9cm}, clip]{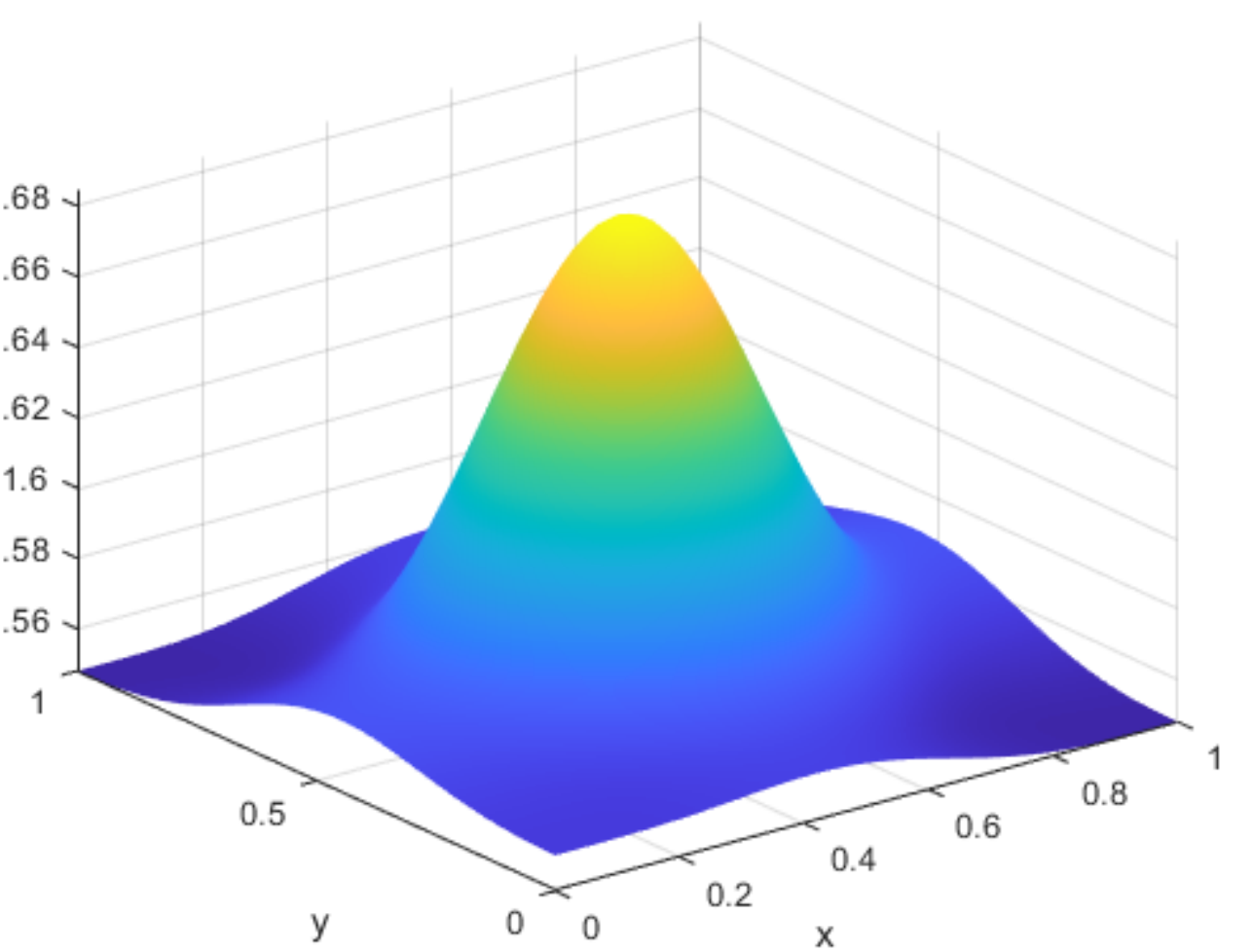}}
\subfigure[$\epsilon=1.00${\rm e-4}]{
\includegraphics[width=0.3\textwidth,trim={4.5cm 9cm 4cm 9cm}, clip]{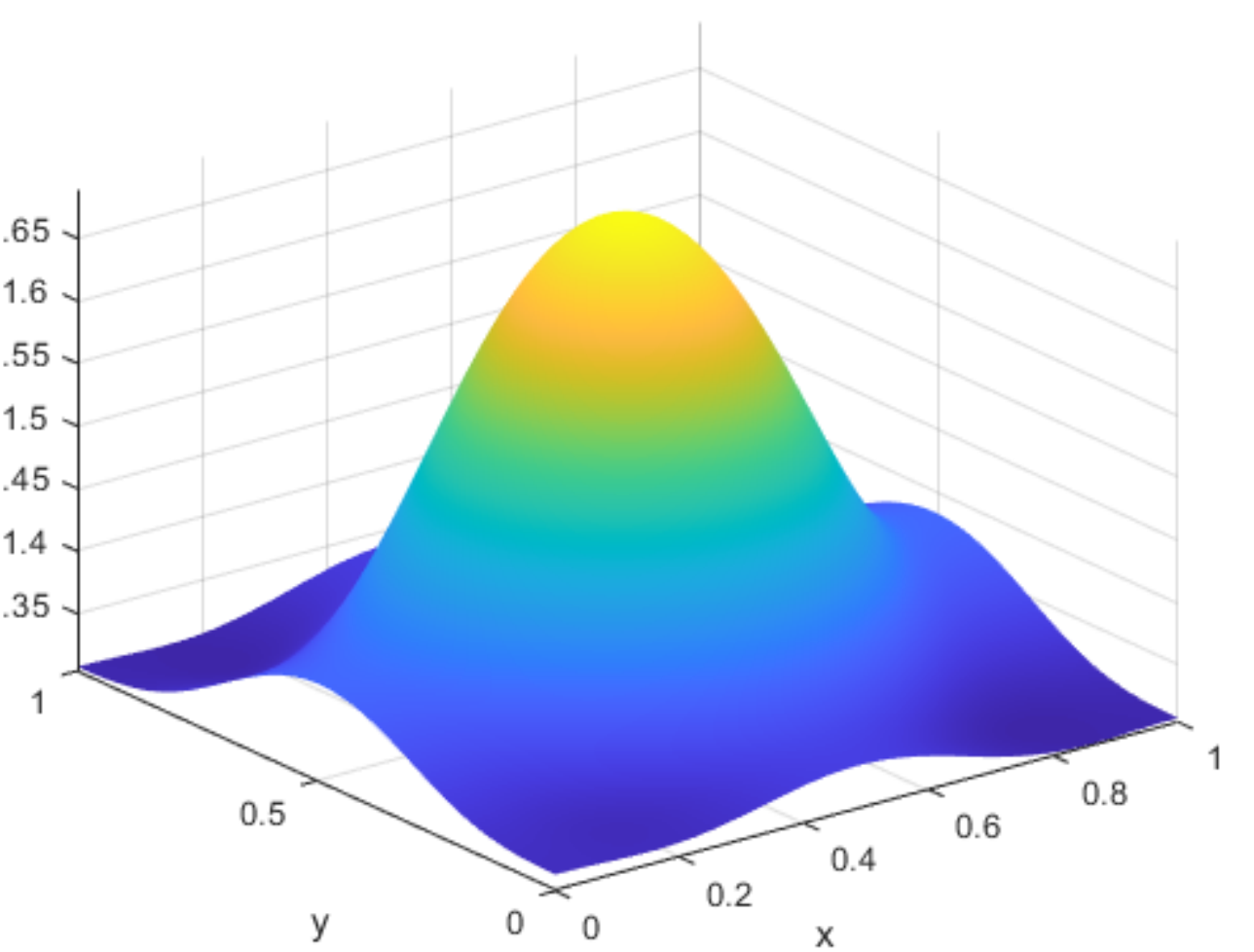}}
\caption{Numerical reconstructions for Example \ref{exam:para}(b).}
	\label{fig:recon-para2d}
\end{figure}

\appendix 
\section{Proofs of technical lemmas}

In this appendix, we collect the proof of Lemmas \ref{lemma:error estimates of u(q)-u_h(q) in L^2 in parabolic system}, \ref{lemma:error estimates of u(q)-u_h(Pi_hq) in L^2 in parabolic system} and \ref{lemma:error estimate of u(q)-u_h(q_h^*) and grad q_h^* in L^2 in parabolic system}. First we prove Lemma \ref{lemma:error estimates of u(q)-u_h(q) in L^2 in parabolic system}.
\begin{proof}
Let $u\equiv u(q^\dag)$ and $u_h^n\equiv u_h^n(q^\dag)$. Under the data regularity in Assumption \ref{assum:Assumption in parabolic system}, we have  $u\in H^1(0,T;H^1_0(\Omega))$. Consequently, 
\begin{equation*}
    |u^n-\bar{u}^n|= \Big|\tau^{-1}\int_{t^{n-1}}^{t^n}\int_{t}^{t^n}\partial_su(s)\mathrm{d}s\mathrm{d}t\Big|\leq\tau^{\frac{1}{2}}\Big(\int_{t^{n-1}}^{t^n}|\partial_tu|^2\mathrm{d}t\Big)^{\frac{1}{2}},
\end{equation*}
and thus the following \textit{a priori} estimate holds
\begin{equation*}
    \tau\sum_{n=1}^N\|u^n-\bar{u}^n\|^2_{L^2(\Omega)} 
    \leq \tau^2\sum_{n=1}^N\|\partial_tu\|^2_{L^2(t^{n-1},t^n;L^2(\Omega))}=\tau^2\|\partial_tu\|^2_{L^2(0,T;L^2(\Omega))}\leq C\tau^2. 
\end{equation*}
Thus it suffices to estimate $\tau\sum_{n=1}^{N}\|\bar{u}^n-u_h^n\|^2_{L^2(\Omega)}$. We introduce a discrete dual problem: find $\{w_h^{n-1}\}_{n=N}^{1}\subset V_{h0}$ with $w_h^N=0$ such that
\begin{equation}\label{equ:dual discrete problem in parabolic system}
		(-\partial_\tau w_h^n,\varphi_h)+(\nabla w_h^{n-1},\nabla\varphi_h)+(q^\dagger w_h^{n-1},\varphi_h) = (\bar{u}^n-u_h^n,\varphi_h), \quad\forall\varphi_h\in V_{h0}.
	\end{equation}
By letting $\varphi_h=-\tau\partial_\tau w_h^n$ in \eqref{equ:dual discrete problem in parabolic system}, and the Cauchy-Schwarz inequality, we have
\begin{align*}
    &\tau\|\partial_\tau w_h^{n}\|^2_{L^2(\Omega)}+\tfrac{1}{2}(\|\nabla w_h^{n-1}\|_{L^2(\Omega)}^2-\|\nabla w_h^{n}\|_{L^2(\Omega)}^2+\|q^{\dag\frac12} w_h^{n-1}\|_{L^2(\Omega)}^2-\|q^{\dag\frac12}w_h^{n}\|_{L^2(\Omega)}^2)\\
    &\leq C\tau\|\bar{u}^n-u_h^n\|_{L^2(\Omega)}^2+\tfrac{\tau}{2}\|\partial_\tau w_h^{n}\|^2_{L^2(\Omega)}.
\end{align*}
Since $w_h^N=0$ and $\nabla w_h^N={ 0}$, summing the inequality over $n$ from $1$ to $N$ gives
\begin{equation}\label{inequ:estimation for partialtau wh}
    \tau\sum_{n=1}^{N}\|\partial_\tau w_h^{n}\|^2_{L^2(\Omega)}\leq  C\tau\sum_{n=1}^{N}\|\bar{u}^n-u_h^n\|_{L^2(\Omega)}^2.
\end{equation}
Similarly, by letting $\varphi_h= 2w_h^{n-1}$ in \eqref{equ:dual discrete problem in parabolic system}, and applying the Cauchy-Schwarz inequality, we deduce
\begin{equation*}
    \tau^{-1}(\|w_h^{n-1}\|^2_{L^2(\Omega)}-\|w_h^{n}\|^2_{L^2(\Omega)})\leq \|\bar{u}^n-u_h^n\|_{L^2(\Omega)}^2+\|w_h^{n-1}\|^2_{L^2(\Omega)}.
\end{equation*}
Then summing this inequality over $n$ from $N$ to any  $k\in\{N,N-1\dots,1\}$ yields
\begin{equation*}
    \|w_h^{k-1}\|^2_{L^2(\Omega)}\leq \tau\sum_{n=N}^{k}\|w_h^{n-1}\|^2_{L^2(\Omega)}+\tau\sum_{n=N}^{k}\|\bar{u}^n-u_h^n\|_{L^2(\Omega)}^2.
\end{equation*}
Then for small $\tau$, the discrete Gronwall's inequality leads to
\begin{equation}\label{inequ:estimation for wh}
    \|w_h^{k-1}\|^2_{L^2(\Omega)}\leq C\tau\sum_{n=1}^{N}\|\bar{u}^n-u_h^n\|_{L^2(\Omega)}^2,
    \quad \forall k\in\{N,N-1,\dots,1\}.
\end{equation}
The estimates \eqref{inequ:estimation for partialtau wh} and \eqref{inequ:estimation for wh} together imply 
\begin{equation}\label{inequ:inequality in dual discrete problem of parabolic system}
    \max_{1\leq  n\leq  N}\|w_h^{n-1}\|_{L^2(\Omega)}^2+\tau\sum_{n=1}^{N}\|\partial_\tau w_h^{n}\|^2_{L^2(\Omega)}\leq  C\tau\sum_{n=1}^{N}\|\bar{u}^n-u_h^n\|_{L^2(\Omega)}^2.
\end{equation}
Next, using the Ritz projection $R_h$, we have the following splitting
    \begin{align*}
    	\tau\sum_{n=1}^{N}\|\bar{u}^n-u_h^n\|_{L^2(\Omega)}^2=&\tau\sum_{n=1}^{N}(\bar{u}^n-u_h^n,\bar{u}^n-R_h\bar{u}^n)+\tau\sum_{n=1}^{N}(\bar{u}^n-u_h^n,R_h\bar{u}^n-u_h^n).
    \end{align*}
Substituting $\varphi_h=\tau(R_h\bar{u}^n-u_h^n)$ into \eqref{equ:dual discrete problem in parabolic system}, using the definition of Ritz projection $R_h$ and summing over $n$ from $1$ to $N$ lead to
\begin{align}\label{equ:bar u^n-u_h^n in L2 norm}
    \tau\sum_{n=1}^{N}\|\bar{u}^n-u_h^n\|_{L^2(\Omega)}^2
    &=\tau\sum_{n=1}^{N}(\bar{u}^n-u_h^n,\bar{u}^n-R_h\bar{u}^n)+\tau\sum_{n=1}^{N}(-\partial_\tau w_h^n,R_h\bar{u}^n-u^n)\nonumber \\
    &\qquad+\tau\sum_{n=1}^{N}(-\partial_\tau w_h^n,u^n-u_h^n)+\tau\sum_{n=1}^{N}(\nabla w_h^{n-1},\nabla (\bar{u}^n-u_h^n))  \nonumber \\
    &\qquad+\tau\sum_{n=1}^N(q^\dagger w_h^{n-1},\bar{u}^n-u_h^n)=: \sum_{i=1}^5{\rm I_{i}}. \nonumber
    \end{align}
Thus, it suffices to bound the summands ${\rm I}_i$, $i=1,\ldots,5$. Meanwhile, integrating the identity \eqref{equ:variational problem in parabolic system} on the interval $(t^{n-1},t^n)$ for $n=1,\ldots,N$ gives
    \begin{equation}\label{equ:integration of variational formulation in papabolic system}
    	(\partial_\tau u^n,\varphi)+(\nabla\bar{u}^n,\nabla\varphi)+(q^\dagger\bar{u}^n,\varphi) =(\bar{f}^n,\varphi), \quad\forall\varphi\in H_0^1(\Omega).
    \end{equation}
Letting $\varphi=w_h^{n-1}$ in \eqref{equ:integration of variational formulation in papabolic system} and $\varphi_h=w_h^{n-1}$ in \eqref{equ:finite element problem in parabolic system} (associated with $q^\dag$ instead of $q_h$, i.e., the finite element problem for $u_h^n$), we deduce 
\begin{align}\label{equ:substracting formulation}
	\tau\sum_{n=1}^N(\partial_\tau(u^n-u_h^n),w^{n-1}_h)&+\tau\sum_{n=1}^N(\nabla(\bar{u}^n-u_h^n),\nabla w^{n-1}_h)\\&+\tau\sum_{n=1}^N(q^\dagger(\bar{u}^n-u_h^n),w^{n-1}_h)=
	\tau\sum_{n=1}^N(\bar{f}^n-f^n,w_h^{n-1}).\nonumber
\end{align}
Meanwhile, since $w_h^N=0$ and $u^0_h=P_hu_0$, by the summation by parts formula, we deduce \begin{equation}\label{equ:the summation by parts formula}
{\rm I_3}=(u_0-P_hu_0,w_h^0)+\tau\sum_{n=1}^N(\partial_\tau(u^n-u_h^n),w^{n-1}_h).
\end{equation}
It follows from \eqref{equ:substracting formulation} and \eqref{equ:the summation by parts formula} that
\begin{equation*}
    \sum_{i=3}^5{\rm I_i}=(u_0-P_hu_0,w_h^0)+\tau\sum_{n=1}^N(\bar{f}^n-f^n,w_h^{n-1}).
\end{equation*}
It remains to bound the terms separately. First, by \eqref{inequ:inequality in dual discrete problem of parabolic system},
\begin{align*}
    |(u_0-P_hu_0,w_h^0)| & \leq \|u_0-P_hu_0\|_{L^2(\Omega)}\|w_h^0\|_{L^2(\Omega)} \leq	Ch^2\|u_0\|_{H^2(\Omega)}\Big(\tau\sum_{n=1}^{N}\|\bar{u}^n-u_h^n\|_{L^2(\Omega)}^2\Big)^{\frac{1}{2}}.
\end{align*}
Meanwhile, since $\tau\sum_{n=1}^N\|\bar f^n-f^n\|_{L^2(\Omega)}^2\leq \tau^2\|\partial_tf\|_{L^2(0,T;L^2(\Omega))}$, we obtain from \eqref{inequ:inequality in dual discrete problem of parabolic system} that
\begin{align*}
    \Big|\tau\sum_{n=1}^N(\bar{f}^n-f^n,w_h^{n-1})\Big| 
    	&\leq 
    	\Big(\tau\sum_{n=1}^N\|\bar{f}^n-f^n\|_{L^2(\Omega)}^2\Big)^{\frac{1}{2}}\Big(\tau\sum_{n=1}^{N}\|\bar{u}^n-u_h^n\|_{L^2(\Omega)}^2\Big)^{\frac{1}{2}} \\
    	&\leq  C\tau\Big(\tau\sum_{n=1}^{N}\|\bar{u}^n-u_h^n\|_{L^2(\Omega)}^2\Big)^{\frac{1}{2}}.
\end{align*}
Further, it follows directly from \eqref{inequ:R_h on H^2}, the a priori regularity $u\in L^2(0,T;H^2(\Omega))\cap H^1(0,T;L^2(\Omega))$ and the estimate \eqref{inequ:inequality in dual discrete problem of parabolic system} that
\begin{align*}
    |{\rm I_1}|&\leq  Ch^2\|u\|_{L^2(0,T;H^2(\Omega))}\Big(\tau\sum_{n=1}^{N}\|\bar{u}^n-u_h^n\|_{L^2(\Omega)}^2\Big)^{\frac{1}{2}},\\
    	|{\rm I_2}|&= \Big|\tau\sum_{n=1}^{N}(-\partial_\tau w_h^n,R_h\bar{u}^n-\bar{u}^n) +\tau\sum_{n=1}^{N}(-\partial_\tau w_h^n,\bar{u}^n-u^n)\Big| \\
        &\leq 
        Ch^2\tau\sum_{n=1}^N\|\partial_\tau w_h^n\|_{L^2(\Omega)}\|\bar{u}^n\|_{H^2(\Omega)}+\tau\sum_{n=1}^N\|\partial_\tau w_h^n\|_{L^2(\Omega)}\|\bar{u}^n-u^n\|_{L^2(\Omega)} \nonumber\\
        &\leq  C(\tau+h^2)(\tau\sum_{n=1}^{N}\|\bar{u}^n-u_h^n\|_{L^2(\Omega)}^2)^{\frac{1}{2}}. \nonumber
    \end{align*}
These estimates together imply
\begin{equation*}
    \tau\sum_{n=1}^{N}\|\bar{u}^n-u_h^n\|^2_{L^2(\Omega)}\leq  C(\tau^2+h^4).
\end{equation*} 
The desired estimate \eqref{inequ:error estimates of u(q)-u_h(q) in L^2 in parabolic system} follows directly by the triangle inequality.
\end{proof}

Next we give the proof of Lemma \ref{lemma:error estimates of u(q)-u_h(Pi_hq) in L^2 in parabolic system}.
\begin{proof}
By Lemma \ref{lemma:error estimates of u(q)-u_h(q) in L^2 in parabolic system}, it suffices to prove
$$\tau\sum_{n=1}^{N}\|u_h^n(q^{\dagger})-u^n_h(\Pi_hq^{\dagger})\|^2_{L^2(\Omega)}\leq  C(\tau^2+h^4).$$ 
Let $\rho_h^n:= u_h^n(q^{\dagger})-u^n_h(\Pi_hq^{\dagger})$ for $n=1,2,\dots,N$. Then $\rho^n_h$ satisfies for any $\varphi_h\in V_{h0}$
	\begin{align}\label{equ:variational formulation for rho_h^n}
		 (\partial_\tau\rho_h^n,\varphi_h)+(\nabla\rho_h,&\nabla\varphi_h)+(\Pi_hq^\dagger\rho_h^n,\varphi_h)=((\Pi_hq^\dagger-q^\dagger)u_h^n(q^\dagger),\varphi_h)\\
		& =((\Pi_hq^\dagger-q^\dagger)(u_h^n(q^\dagger)-u^n(q^\dagger)),\varphi_h)+((\Pi_hq^\dagger-q^\dagger)u^n(q^\dagger),\varphi_h). \nonumber
	\end{align}
Letting $\varphi_h=2\rho_h^n$ in \eqref{equ:variational formulation for rho_h^n}, by \eqref{inequ:pi_h on H^2}, Assumption \ref{assum:Assumption in parabolic system}, the regularity $\|u(q^\dag)\|_{L^\infty(0,T;L^\infty(\Omega))}\leq C$ and Young's inequality, there holds
\begin{align*}
    	\tau^{-1}(\|\rho_h^n\|^2_{L^2(\Omega)}-\|\rho_h^{n-1}\|^2_{L^2(\Omega)}) 
    	&\leq 
    	C\|q^\dagger\|_{L^\infty(\Omega)}\|u^n(q^\dagger)-u_h^n(q^\dagger)\|_{L^2(\Omega)}\|\rho_h^n\|_{L^2(\Omega)}\\
    	&\quad + Ch^2\|q^\dagger\|_{H^2(\Omega)}\|u(q^\dagger)\|_{L^\infty(0,T;L^\infty(\Omega))})\|\rho_h^n\|_{L^2(\Omega)}\\
    	&\leq  C\|u^n(q^\dagger)-u_h^n(q^\dagger)\|^2_{L^2(\Omega)}+Ch^4+\|\rho_h^n\|_{L^2(\Omega)}^2.
\end{align*}
Summing the inequality over $n$ from $1$ to any $k\in\{1,2\dots,N\}$ and noting $\rho_h^0=0$ and Lemma \ref{lemma:error estimates of u(q)-u_h(q) in L^2 in parabolic system}, we have
    \begin{equation*}
        \|\rho_h^k\|_{L^2(\Omega)}^2\leq  C\tau\sum_{n=1}^{k}\|\rho_h^n\|^2_{L^2(\Omega)}+C(\tau^2+h^4),
    \end{equation*}
    and thus for small $\tau$, the discrete Gronwall's inequality leads directly to
    \begin{equation*}
    	\|\rho_h^k\|_{L^2(\Omega)}^2\leq C(\tau^2+h^4),\quad k=1,2,\dots,N.
    \end{equation*}
This completes the proof of the lemma.
\end{proof}

Last we give the proof of Lemma \ref{lemma:error estimate of u(q)-u_h(q_h^*) and grad q_h^* in L^2 in parabolic system}.
\begin{proof}
First we prove an elementary estimate:
\begin{equation}\label{eqn:err-NL}
      \tau\sum_{n=N_{0}}^N\|u^n(q^\dagger)-z^\delta_n\|^2_{L^2(\Omega)}\leq C(\tau^2+\delta^2).
 \end{equation}
Indeed, by the proof of Lemma \ref{lemma:error estimates of u(q)-u_h(q) in L^2 in parabolic system}, we have
$\tau\sum_{n=N_{0}}^N\|u^n(q^{\dagger})-\bar{u}^n(q^{\dagger})\|^2_{L^2(\Omega)} \leq C\tau^2.$ 
Meanwhile, direct computation leads to
\begin{equation*}
    |\bar{u}^n(q^{\dagger})-z_n^\delta|=\Big|\tau^{-1}\int_{t^{n-1}}^{t^n}u(q^\dagger)(t)-z^\delta(t)\mathrm{d}t\Big|\leq\tau^{-\frac{1}{2}}\Big(\int_{t^{n-1}}^{t^n}|u(q^\dagger)(t)-z^\delta(t)|^2\mathrm{d}t\Big)^{\frac{1}{2}},
\end{equation*}
and hence
\begin{equation*}
    \tau\sum_{n=N_{0}}^N\|\bar{u}^n(q^{\dagger})-z_n^\delta\|^2_{L^2(\Omega)}\leq \|u(q^\dagger) -z^\delta \|_{L^2(T_0,T;L^2(\Omega))}^2=\delta^2.
\end{equation*}
The claim \eqref{eqn:err-NL} follows from triangle inequality.
Since $q_h^*$ is the minimizer of the system \eqref{equ:discrete minimizing problem in parabolic system}-\eqref{equ:finite element problem in parabolic system} and $\Pi_hq^\dagger\in K_h$, we have 
$$J_{\alpha,h,\tau}(q_h^*)\leq  J_{\alpha,h,\tau}(\Pi_hq^\dagger).$$ 
Then by the inequality \eqref{inequ:pi_h on H^2}, $\|\nabla\Pi_hq^\dag\|_{L^2(\Omega)}\leq C\| q^\dag\|_{H^1(\Omega)}$, and thus there holds
\begin{align*}
	\tau\sum_{n=N_{0}}^N&\|u_h^n(q^*_h)-z^\delta_n\|^2_{L^2(\Omega)}+\alpha\|\nabla q_h^*\|^2_{L^2(\Omega)} 
		\leq   \tau\sum_{n=N_{0}}^N\|u_h^n(\Pi_hq^\dagger)-z^\delta_n\|^2_{L^2(\Omega)}+\alpha\|\nabla \Pi_hq^\dagger\|^2_{L^2(\Omega)} \\
		\leq  & C\tau\sum_{n=N_{0}}^N\|u_h^n(\Pi_hq^\dagger)-u^n(q^\dagger)\|^2_{L^2(\Omega)}+C\tau\sum_{n=N_{0}}^N\|u^n(q^\dagger)-z^\delta_n\|^2_{L^2(\Omega)}+C\alpha
		\leq   C(\tau^2+h^4+\delta^2+\alpha).
	\end{align*}
where the last step is due to Lemma \ref{lemma:error estimates of u(q)-u_h(Pi_hq) in L^2 in parabolic system} and \eqref{eqn:err-NL}. Then by the triangle inequality and \eqref{eqn:err-NL}, we obtain
	\begin{align*}
	    &\quad\tau\sum_{n=N_{0}}^N\|u^n(q^\dagger)-u^n_h(q_h^*)\|^2_{L^2(\Omega)}+\alpha\|\nabla q_h^*\|^2_{L^2(\Omega)} \\
	    &\leq 
	    C\Big(\tau\sum_{n=N_{0}}^N\|u^n(q^\dagger)-z_n^\delta\|^2_{L^2(\Omega)}+\tau\sum_{n=N_{0}}^N\|z_n^\delta-u^n_h(q_h^*)\|^2_{L^2(\Omega)}+\alpha\|\nabla q_h^*\|^2_{L^2(\Omega)}\Big) \\
	    &\leq C(\tau^2+h^4+\delta^2+\alpha).
	\end{align*}
This completes the proof of the lemma.
\end{proof}

\bibliographystyle{abbrv}
\bibliography{reference}

\begin{thebibliography}{10}

\bibitem{Adams2003Sobolev}
R.~A. Adams and J.~J.~F. Fournier.
\newblock {\em {Sobolev Spaces}}.
\newblock Elsevier/Academic Press, Amsterdam, second edition, 2003.

\bibitem{Alifanov:1995}
O.~M. Alifanov, E.~A. Artyukhin, and S.~V. Rumyantsev.
\newblock {\em Extreme {M}ethods for {S}olving {I}ll-{P}osed {P}roblems with
  {A}pplications to {I}nverse {H}eat {T}ransfer {P}roblems}.
\newblock Begell House, New York, 1995.

\bibitem{bachmayr2019identifiability}
M.~Bachmayr and V.~K. Nguyen.
\newblock Identifiability of diffusion coefficients for source terms of
  nonuniform sign.
\newblock {\em Inverse Probl. Imag.}, 13(5):1007--1021, 2019.

\bibitem{BerettaCavaterra:2011}
E.~Beretta and C.~Cavaterra.
\newblock Identifying a space dependent coefficient in a reaction-diffusion
  equation.
\newblock {\em Inverse Probl. Imaging}, 5(2):285--296, 2011.

\bibitem{Bobkov2018OnMA}
V.~Bobkov and P.~Tak\'{a}\v{c}.
\newblock On maximum and comparison principles for parabolic problems with the
  {$p$}-{L}aplacian.
\newblock {\em Rev. R. Acad. Cienc. Exactas F\'{\i}s. Nat. Ser. A Mat. RACSAM},
  113(2):1141--1158, 2019.

\bibitem{doi:10.1137/16M1094476}
A.~Bonito, A.~Cohen, R.~DeVore, G.~Petrova, and G.~Welper.
\newblock Diffusion coefficients estimation for elliptic partial differential
  equations.
\newblock {\em SIAM J. Math. Anal.}, 49(2):1570--1592, 2017.

\bibitem{Brenner2002The}
S.~C. Brenner and L.~R. Scott.
\newblock {\em {The Mathematical Theory of Finite Element Methods}}.
\newblock Springer-Verlag, New York, second edition, 2002.

\bibitem{Brezis2010FunctionalAS}
H.~Brezis.
\newblock {\em {Functional Analysis, {S}obolev Spaces and Partial Differential
  Equations}}.
\newblock Universitext. Springer, New York, 2011.

\bibitem{BurmanErn:2021}
E.~Burman, G.~Delay, and A.~Ern.
\newblock A hybridized high-order method for unique continuation subject to the
  {H}elmholtz equation.
\newblock {\em SIAM J. Numer. Anal.}, 59(5):2368--2392, 2021.

\bibitem{BurmanOksanen:2020}
E.~Burman, A.~Feizmohammadi, and L.~Oksanen.
\newblock A fully discrete numerical control method for the wave equation.
\newblock {\em SIAM J. Control Optim.}, 58(3):1519--1546, 2020.

\bibitem{Chen2020Convergence}
D.-H. Chen, D.~Jiang, and J.~Zou.
\newblock Convergence rates of {T}ikhonov regularizations for elliptic and
  parabolic inverse radiativity problems.
\newblock {\em Inverse Problems}, 36(7):075001, 21, 2020.

\bibitem{ChengYamamoto:2000}
J.~Cheng and M.~Yamamoto.
\newblock One new strategy for a priori choice of regularizing parameters in
  {T}ikhonov's regularization.
\newblock {\em Inverse Problems}, 16(4):L31--L38, 2000.

\bibitem{Choulli:2020}
M.~Choulli.
\newblock Some stability inequalities for hybrid inverse problems.
\newblock {\em Comptes Rendus. Mathématique}, 359(10):1251--1265, 2021.

\bibitem{ChoulliYamamoto:2008}
M.~Choulli and M.~Yamamoto.
\newblock Uniqueness and stability in determining the heat radiative
  coefficient, the initial temperature and a boundary coefficient in a
  parabolic equation.
\newblock {\em Nonlinear Anal.}, 69(11):3983--3998, 2008.

\bibitem{Ciarlet1991BasicEE}
P.~G. Ciarlet.
\newblock Basic error estimates for elliptic problems.
\newblock In {\em {Handbook of Numerical Analysis, {V}ol. {II}}}, Handb. Numer.
  Anal., II, pages 17--351. North-Holland, Amsterdam, 1991.

\bibitem{DengYuYang:2008}
Z.-C. Deng, J.-N. Yu, and L.~Yang.
\newblock Optimization method for an evolutional type inverse heat conduction
  problem.
\newblock {\em J. Phys. A}, 41(3):035201, 20, 2008.

\bibitem{EggerHofmann:2018}
H.~Egger and B.~Hofmann.
\newblock Tikhonov regularization in {H}ilbert scales under conditional
  stability assumptions.
\newblock {\em Inverse Problems}, 34(11):115015, 17, 2018.

\bibitem{Engl1996Regularization}
H.~W. Engl, M.~Hanke, and A.~Neubauer.
\newblock {\em {Regularization of Inverse Problems}}.
\newblock Kluwer Academic Publishers Group, Dordrecht, 1996.

\bibitem{EnglKunischNeubauer:1989}
H.~W. Engl, K.~Kunisch, and A.~Neubauer.
\newblock Convergence rates for {T}ikhonov regularisation of nonlinear
  ill-posed problems.
\newblock {\em Inverse Problems}, 5(4):523--540, 1989.

\bibitem{Evans2010PartialDE}
L.~C. Evans.
\newblock {\em {Partial Differential Equations}}.
\newblock American Mathematical Society, Providence, RI, second edition, 2010.

\bibitem{Friedman:1958}
A.~Friedman.
\newblock Remarks on the maximum principle for parabolic equations and its
  applications.
\newblock {\em Pacific J. Math.}, 8:201--211, 1958.

\bibitem{Gilbarg1977EllipticPD}
D.~Gilbarg and N.~S. Trudinger.
\newblock {\em {Elliptic Partial Differential Equations of Second Order}}.
\newblock Grundlehren der Mathematischen Wissenschaften, Vol. 224.
  Springer-Verlag, Berlin-New York, 1977.

\bibitem{GruterWidman:1982}
M.~Gr\"{u}ter and K.-O. Widman.
\newblock The {G}reen function for uniformly elliptic equations.
\newblock {\em Manuscripta Math.}, 37(3):303--342, 1982.

\bibitem{HaoQuyen:2010}
D.~N. H\`ao and T.~N.~T. Quyen.
\newblock Convergence rates for {T}ikhonov regularization of coefficient
  identification problems in {L}aplace-type equations.
\newblock {\em Inverse Problems}, 26(12):125014, 23, 2010.

\bibitem{hao2014finite}
D.~N. H{\`a}o and T.~N.~T. Quyen.
\newblock Finite element methods for coefficient identification in an elliptic
  equation.
\newblock {\em Appl. Anal.}, 93(7):1533--1566, 2014.

\bibitem{ItoJin:2015}
K.~Ito and B.~Jin.
\newblock {\em Inverse {P}roblems: Tikhonov {T}heory and {A}lgorithms}.
\newblock World Scientific Publishing Co. Pte. Ltd., Hackensack, NJ, 2015.

\bibitem{Jin2021EA}
B.~Jin and Z.~Zhou.
\newblock Error analysis of finite element approximations of diffusion
  coefficient identification for elliptic and parabolic problems.
\newblock {\em SIAM J. Numer. Anal.}, 59(1):119--142, 2021.

\bibitem{KamyninKostin:2010}
V.~L. Kamynin and A.~B. Kostin.
\newblock Two inverse problems of the determination of a coefficient in a
  parabolic equation.
\newblock {\em Differ. Uravn.}, 46(3):372--383, 2010.

\bibitem{KlibanovLiZhang:2020}
M.~V. Klibanov, J.~Li, and W.~Zhang.
\newblock Convexification for an inverse parabolic problem.
\newblock {\em Inverse Problems}, 36(8):085008, 32, 2020.

\bibitem{kohn1988variational}
R.~V. Kohn and B.~D. Lowe.
\newblock A variational method for parameter identification.
\newblock {\em ESAIM: Math. Model. Numer. Anal.}, 22(1):119--158, 1988.

\bibitem{LittmanStampacchia:1963}
W.~Littman, G.~Stampacchia, and H.~F. Weinberger.
\newblock Regular points for elliptic equations with discontinuous
  coefficients.
\newblock {\em Ann. Scuola Norm. Sup. Pisa Cl. Sci. (3)}, 17:43--77, 1963.

\bibitem{Nirenberg:1953}
L.~Nirenberg.
\newblock A strong maximum principle for parabolic equations.
\newblock {\em Comm. Pure Appl. Math.}, 6:167--177, 1953.

\bibitem{Pennes:1948}
H.~H. Pennes.
\newblock Analysis of tissue and arterial blood temperatures in the resting
  human forearm.
\newblock {\em J. Appl. Physiol.}, 1(2):93--122, 1948.

\bibitem{PrilepkoKostin:1992}
A.~I. Prilepko and A.~B. Kostin.
\newblock Some inverse problems for parabolic equations with final and integral
  observation.
\newblock {\em Mat. Sb.}, 183(4):49--68, 1992.

\bibitem{ScottRobinson:1998}
E.~P. Scott, P.~S. Robinson, and T.~E. Diller.
\newblock Development of methodologies for the estimation of blood perfusion
  using a minimally invasive thermal probe.
\newblock {\em Meas. Sci. Technol.}, 9(6):888--897, 1998.

\bibitem{Thome2006GalerkinFE}
V.~Thom\'{e}e.
\newblock {\em {Galerkin Finite Element Methods for Parabolic Problems}}.
\newblock Springer-Verlag, second edition, 2006.

\bibitem{TrucuInghamLesnic:2010}
D.~Trucu, D.~B. Ingham, and D.~Lesnic.
\newblock Space-dependent perfusion coefficient identification in the transient
  bio-heat equation.
\newblock {\em J. Engrg. Math.}, 67(4):307--315, 2010.

\bibitem{Vzquez1984ASM}
J.~L. V\'{a}zquez.
\newblock A strong maximum principle for some quasilinear elliptic equations.
\newblock {\em Appl. Math. Optim.}, 12(3):191--202, 1984.

\bibitem{Wang2010ErrorEO}
L.~Wang and J.~Zou.
\newblock Error estimates of finite element methods for parameter
  identifications in elliptic and parabolic systems.
\newblock {\em Discrete Contin. Dyn. Syst. Ser. B}, 14(4):1641--1670, 2010.

\bibitem{WernerHofmann:2020}
F.~Werner and B.~Hofmann.
\newblock Convergence analysis of (statistical) inverse problems under
  conditional stability estimates.
\newblock {\em Inverse Problems}, 36(1):015004, 23, 2020.

\bibitem{YamamotoZou:2001}
M.~Yamamoto and J.~Zou.
\newblock Simultaneous reconstruction of the initial temperature and heat
  radiative coefficient.
\newblock {\em Inverse Problems}, 17(4):1181--1202, 2001.

\bibitem{YangYuDeng:2008}
L.~Yang, J.-N. Yu, and Z.-C. Deng.
\newblock An inverse problem of identifying the coefficient of parabolic
  equation.
\newblock {\em Appl. Math. Model.}, 32(10):1984--1995, 2008.

\bibitem{YueZhangZuo:2008}
K.~Yue, X.~Zhang, and Y.~Y. Zuo.
\newblock Noninvasive method for simultaneously measuring the thermophysical
  properties and blood perfusion in cylindrically shaped living tissues.
\newblock {\em Cell Biochem. Biophys.}, 50(1):41–--51, 2008.

\end{thebibliography}

\end{document}